\documentclass[a4paper,twoside]{article}      % Comments after  % are ignored
\usepackage{amsmath,amssymb,amsfonts,amsthm,amscd,mathtools}

\usepackage{graphics}                 % Packages to allow inclusion of graphics
\usepackage{color}                    % For creating coloured text and background
\usepackage{hyperref}                % For creating hyperlinks in cross references
\usepackage{ mathrsfs }
\usepackage{indentfirst}
\usepackage{bbold}
\usepackage{bm}
\usepackage{caption}
\usepackage{subfigure}
\usepackage{enumerate}
\usepackage{url}         %to include url's in citations
\usepackage{colonequals} %to use the := symbol
\usepackage{authblk} % different authors with same affiliation
\usepackage{a4wide}
\usepackage{tikz}
\usepackage{graphicx}
\hypersetup{
    unicode=false,          % non-Latin characters in Acrobat‚????s bookmarks
    pdftoolbar=true,        % show Acrobat‚????s toolbar?
    pdfmenubar=true,        % show Acrobat‚????s menu?
    pdffitwindow=false,     % window fit to page when opened
    pdfstartview={FitH},    % fits the width of the page to the window
    pdftitle={My title},    % title
    pdfauthor={Author},     % author
    pdfsubject={Subject},   % subject of the document
    pdfcreator={Creator},   % creator of the document
    pdfproducer={Producer}, % producer of the document
    pdfkeywords={keywords}, % list of keywords
    pdfnewwindow=true,      % links in new window
    colorlinks=true,       % false: boxed links; true: colored links
    linkcolor=blue,          % color of internal links
    citecolor=blue,        % color of links to bibliography
    filecolor=magenta,      % color of file links
    urlcolor=cyan           % color of external links
}

\newtheorem{theorem}{Theorem}[section]
\newtheorem{proposition}[theorem]{Proposition}

\newtheorem{lemma}[theorem]{Lemma}
\theoremstyle{definition}
\newtheorem{remark}[theorem]{Remark}

\def\!{\mathop{\mathrm{!}}}

\def\x{\textbf{x}}

\def\f{\bar{f}}

\newlength{\boxwidth}
\author[1]{N. Balabanova}
\author[1]{M. H. Duong}
\author[2]{T. A. Han}
\affil[1]{School of Mathematics, University of Birmingham, UK.}
\affil[2]{School of Computing and Digital Technologies, Teesside University, UK. }
\title{Replicator-mutator dynamics for public goods games with institutional incentives}
\begin{document}
\maketitle
\begin{abstract}
Understanding the emergence and stability of cooperation in public goods games is  important due to its applications in fields such as biology, economics, and social science. However, a gap remains in comprehending how mutations, both additive and multiplicative, as well as institutional incentives, influence these dynamics. 
In this paper, we study the replicator-mutator dynamics, with combined additive and multiplicative mutations, for public goods games both in the absence or presence of institutional incentives.  For each model, we identify the possible number of (stable) equilibria, demonstrate their attainability, as well as analyse their stability properties. We also characterise the dependence of these equilibria on the model's parameters via bifurcation analysis and asymptotic behaviour. Our results offer rigorous and quantitative insights into the role of institutional incentives and  the effect of combined additive and multiplicative mutations on the evolution of cooperation in the context of public goods games. \end{abstract}
\tableofcontents

\section{Introduction} 

%Applications Across Disciplines:
%Evolutionary Biology: PGGs help explain cooperative behaviors in microbial populations, where individuals secrete beneficial enzymes at a personal cost (West et al., 2007).
%Economics and Social Science: PGG experiments provide insights into human cooperation, informing policies on tax compliance, climate change mitigation, and collective governance (Ostrom, 1990).
%Artificial Intelligence and Multi-Agent Systems: Understanding cooperative dynamics aids in designing decentralized AI systems that require coordination and resource sharing (Perc et al., 2017).
%By integrating game-theoretic models with experimental and empirical research, PGGs continue to be a cornerstone in studying how cooperation can emerge and be maintained despite selfish incentives.

\subsection{The replicator-mutator dynamics}
The replicator-mutator dynamics, a fundamental model in evolutionary game theory, is a system of differential equations describing the evolution dynamics in a population of multiple co-present strategies where selection and mutation are both captured. It provides a powerful mathematical framework for the modelling and analysis of complex biological, economical and social systems.  It has been employed in the study of, among other applications, population genetics \cite{Hadeler1981}, autocatalytic reaction networks
\cite{StadlerSchuster1992}, language evolution
\cite{Nowaketal2001}, the evolution of cooperation \cite{Imhof-etal2005} and dynamics of behavior in social networks \cite{Olfati2007}. 

To describe the model, we consider an infinite population consisting of $n$ types/strategies $S_1,\cdots, S_n$ whose frequencies are, respectively, $x_1,\cdots, x_n$. These types undergo selection; that is, the reproduction rate of each type, $S_i$, is determined by its fitness or average payoff, $f_i$, which is obtained from interacting with other individuals in the population. The interaction of the individuals in the population is carried out within randomly selected groups of $d$ participants (for some integer $d$). That is, they play and obtain their payoffs from a $d$-player game, defined by a payoff matrix. Furthermore, they are also subjected to \textit{multiplicative and additive mutations}. The multiplicative mutation is described by  an $n\times n$ row-stochastic matrix (known as mutation matrix) $Q=(q_{ji}), j,i\in\{1,\cdots,n\}$, that is
\begin{equation}
\sum_{j=1}^n q_{ji}=1, \quad 1\leq i\leq n.    
\label{eq: sum of qij}
\end{equation}
The entry $q_{ji}$ often models the probability that the offsprings of the $j$-th type that resembles the $i$-th type due to error in the replication process. On the other hand, the additive mutation is characterized by a non-negative rate, $\mu$, at which a player changes its strategy.

%the  probability that a player of type $S_j$ changes its type or strategy to $S_i$. The mutation matrix $Q$ is a row-stochastic matrix, i.e.,

%\textcolor{red}{role of additive and multiplicative mutation}

%Mutation is included by adding the possibility that individuals spontaneously change from one strategy to another, which is modeled via a mutation matrix, $Q=(q_{ji}), j,i\in\{1,\cdots,n\}$. The entry $q_{ji}$ denotes the  probability that a player of type $S_j$ changes its type or strategy to $S_i$. The mutation matrix $Q$ is a row-stochastic matrix, i.e.,
%\begin{equation}
%\sum_{j=1}^n q_{ji}=1, \quad 1\leq i\leq n.    
%\label{eq: sum of qij}
%\end{equation}
The resulting replicator-mutator dynamics, which incorporate both additive and multiplicative mutations, is then given by, see e.g. \cite{nowak2001evolution,Komarova2001JTB,Komarova2004, Komarova2010, Pais2012, mittal2020,chakraborty2024replicator, mukhopadhyay2021chaos} 
\begin{equation}
\label{eq: RME}
\dot{x}_i=\sum_{j=1}^n x_j f_j(\x)q_{ji}- x_i \f(\x)-\mu (n x_i-1),\qquad  i=1,\ldots, n,
\end{equation}
where $\x = (x_1, x_2, \dots, x_n)$ and $\f(\x)=\sum_{i=1}^n x_i f_i(\x)$ denotes the average fitness of the whole population.  The replicator dynamics is a special instance of \eqref{eq: RME} when the mutation matrix is the identity matrix, $Q=I$, and there is no additive mutation $\mu=0$. Note that the specific form of \eqref{eq: RME} ensures that the simplex
\[
S_n:=\{x=(x_1,\ldots, x_n)\in \mathbb{R}^n: \sum_{i=1}^nx_i=1\quad\text{and}\quad x_i\geq 0 \quad\text{for}\quad i=1,\ldots, n\}
\]
is invariant under \eqref{eq: RME} provided that initially it belongs to this set (see Section \ref{sec: invariant} for a proof), thus it is biologically consistent since these strategies comprise the whole population.  In the next section, we employ the replicator-mutator dynamics \eqref{eq: RME} in the study of evolution and promotion of cooperation, where the players involve in the so-called Public Goods Games.
\subsection{Public Goods Games and the evolution of cooperation}
The evolution of cooperation is a central question in evolutionary biology, economics, and social science. Public Goods Games (PGGs) provide a powerful framework to study this phenomenon by modelling situations where individuals must choose between contributing to a collective resource or free-riding on the contributions of others. These games capture the fundamental social dilemma where individual rationality leads to collective sub-optimality \cite{hardin:1968mm}. In a typical PGG, each participant decides how much to contribute to a shared resource. The total contributions are multiplied by a factor greater than one, representing the benefits of cooperation, and then distributed equally among all players, regardless of their individual contributions \cite{ledyard:1995aa}. While cooperation benefits the group, individuals have an incentive to withhold contributions to maximize personal gain, leading to the risk of cooperation breakdown \cite{olson:1971aa}.

PGGs are essential for understanding the mechanisms that sustain cooperation in human societies and biological populations. Several mechanisms supporting cooperation have been recognized such as direct and indirect reciprocity (cooperation can evolve when individuals interact repeatedly \cite{hilbe2018indirect,xia2023reputation} or when reputation plays a role \cite{nowak1998evolution,xia2023reputation}), network and group structures (the spatial arrangement of individuals can promote cooperation by enabling clusters of cooperators to form and persist \cite{santos2005scale,szabo2007evolutionary}) and institutional incentives (the introduction of costly punishment for defectors or rewards for contributors can sustain cooperation)  \cite{sigmund2001reward,han2024evolutionary,chen2015first,duong2021cost}. The latter mechanism is the focus of the present paper.

\subsection{Literature Review}
As has been mentioned earlier, the replicator-mutator dynamics arise naturally in the study of many complex systems and phenomena in evolutionary biology, social sciences and evolution of languages. Thus, there is now a large body of research on this important model, see for instance the aforementioned papers \cite{nowak2001evolution,Komarova2001JTB,Komarova2004, Komarova2010, Pais2012,mittal2020, mukhopadhyay2021chaos,chakraborty2024replicator} for deterministic games (the last three papers capture both additive and multiplicative mutations) and \cite{DuongHan2019,DuongHanDGA2020,DuongHan2021,chen2024number} for random games, and references therein. We also refer the reader to some recent papers that extend the replicator-mutator dynamics, such as \cite{bauer2019stabilization} to multi-population models and \cite{du2024replicator} to take into account the environmental feedbacks.

In the following, we will review in more detail the most relevant papers that employ the replicator-mutator dynamics with institutional incentives.  

In \cite{dong2019competitive}, the authors evaluate the effectiveness of three institutional incentives—reward, punishment, and a combination of both—by examining group fitness at stable evolutionary equilibria. Their results show that the optimal incentive depends on decision errors (modelling as additive mutation). Without errors, a mix of reward and punishment maximizes cooperation and fitness. However, as errors increase, reward becomes the most effective, while punishment should be avoided. The paper \cite{jia2024synergistic} investigates the synergistic effects of global exclusion and mutation on replicator dynamics of public cooperation. It shows that the replicator-mutation dynamics can result in either a global stable coexistence or two local stable coexistences, whose attraction basins are separated by an unstable fixed point, between global exclusion and defection, as well as several types of bifurcations. In \cite{chiba2024can}, the authors study the evolution of cooperation in the presence of an institution that is autonomous, in the sense that it has its own interests that may or may not align with those of the population. By comparing the efficiencies of three types of institutional incentives (namely, reward, punishment, and a mixture of reward and punishment) and analysing the group fitness at the stable equilibria of evolutionary dynamics, the authors highlight the potential benefits of institutional wealth redistribution.

%study promotion of cooperation by wealth redistrubution. 
%that study the evolution of cooperation using the replicator (no mutation) dynamics for PGGs and other games, see for instance \cite{hauert2004dynamics,hauert2006evolutionary, sasaki2011replicator} 
    
\subsection{Summary of main results}
The key novelty of the present work is a mathematically rigorous and quantitative characterisation of how institutional incentives and of the combination of additive and multiplicative mutations influence the evolution  of cooperation in the context of the PGGs. This research goes beyond the existing aforementioned works in this direction \cite{dong2019competitive,jia2024synergistic,chiba2024can}, which  take consider only additive mutations.

More precisely, to highlight the role of institutional incentive, we consider the replicator-mutator dynamics both in the absence or presence of the incentive. For each model, we identify the possible number of (stable and unstable) equilibria, demonstrate their attainability, as well as analyse their stability properties. We also characterise the dependence of these equilibria on the model's parameters via bifurcation analysis and asymptotic behaviour. A thorough mathematical analysis of these models is highly non-trivial since they are nonlinear differential equations and contain many biologically relevant parameters. We overcome the challenge by making use of the symmetries arising from the models. Precise statements of the main results will be presented in Section \ref{sec: results}.

\subsection{Organisation of the paper} 
The rest of the paper is organised as follows. In Section \ref{sec: models}, we describe the models and methods. In Section \ref{sec: results}, we present the main results and their interpretations. A summary of the paper and further discussion on future research are given in Section \ref{sec: summary}. Finally, detailed and technical proofs are postponed to the appendix.
\section{Models and Methods}
\label{sec: models}
In this section, we describe the mathematical models that will be studied. We start by applying the general replicator-mutator dynamics \eqref{eq: RME} to obtain a general multi-player two-strategy games, which will be used to study the evolution of cooperation (and defection). We then further employ this framework to public goods games with and without institutional incentives. 

\subsection{Muti-player two-strategy games}
In this section, we focus on the replicator-mutator equation for $d$-player two-strategy games with a symmetric mutation matrix $Q=(q_{ji})$ (with  $j,i\in\{1,2\}$) so that
\[
q_{11}=q_{22}=1-q \quad \text{and}\quad q_{12}=q_{21}=q,
\]
for some constant $0\leq q\leq 1/2$. Let $x$ be the frequency of $S_1$ (cooperation). Thus the frequency of $S_2$ (defection) is $1-x$. The interaction of the individuals in the population is in randomly selected groups of $d$ participants, that is, they play and obtain their fitness from  $d$-player games. Let $a_k$ (resp., $b_k$) be the payoff of an $S_1$-strategist (resp., $S_2$) in a group  containing  other $k$ $S_1$ strategists (i.e. $d-1-k$ $S_2$ strategists). Here we consider symmetric games where the payoffs do not depend on the ordering of the players. In this case, the average payoffs of $S_1$ and $S_2$ are, respectively 
\begin{equation}
\label{eq: fitness}
f_1(x)= \sum\limits_{k=0}^{d-1}a_k\begin{pmatrix}
d-1\\
k
\end{pmatrix}x^k (1-x)^{d-1-k}\quad\text{and}\quad
f_2(x)= \sum\limits_{k=0}^{d-1}b_k\begin{pmatrix}
d-1\\
k
\end{pmatrix}x^k (1-x)^{d-1-k}.
\end{equation} 
The replicator-mutator equation \eqref{eq: RME} then becomes
\begin{align}
\label{eq: RME2}
\dot{x}&=x f_1(x)(1-q)+(1-x) f_2(x)q-x(x f_1(x)+(1-x)f_2(x))-\mu(2x-1)\notag
\\&=q\Big[(1-x)f_2(x)-x f_1(x)\Big]+x(1-x)(f_1(x)-f_2(x))-\mu(2x-1).
\end{align}
Note that when $q=0$ and $\mu=0$, we recover the  replicator equation (i.e. without mutation). In contrast to the replicator equation, $x=0$ and $x=1$ are no longer equilibrium points of the system for $q\neq 0$ or $\mu\neq 0$. In addition, from \eqref{eq: RME2} it follows that, if $q=\frac{1}{2}$ then $x=\frac{1}{2}$ is always an equilibrium point.

Equilibrium points are those points $0\leq x\leq 1$ that make the right-hand side of \eqref{eq: RME2} vanish, that is
\begin{equation}
\label{eq: equilibria1}
G(x):=q\Big[(1-x)f_2(x)-x f_1(x)\Big]+x(1-x)(f_1(x)-f_2(x))-\mu(2x-1)=0.
\end{equation}
Using \eqref{eq: fitness}, $G$ becomes 
\begin{equation}
\label{eq: G}
G(x)=G_1(x)-\mu(2x-1),
\end{equation}
where
\begin{equation}
\label{eq: equilibria2}
G_1(x)=\sum_{k=0}^{d-1}\begin{pmatrix}
d-1\\
k
\end{pmatrix}x^k (1-x)^{d-1-k}\Big[a_k x(1-x-q)-b_k(1-x)(x-q)\Big].
\end{equation}
In the next sections, we consider Public Goods Games (PGGs) with and without institutional incentives where explicit formulas for the payoff entries $\{a_k, b_k\}_{k=0}^{d-1}$ will be given. 
\subsection{PGGs without institutional incentives}
\label{sec: no institutional incentive}
We start with the case of PGGs without institutional incentives. We then study the combined influence of the additive and multiplicative mutations on the equilibrium properties of the replicator-mutator dynamics.

Let $c$ be the cost of contribution, and $r$ be the multiplication factor ($1 < r < d$) \cite{hauert:2007aa}. For $0\leq k\leq d-1$, let $a_k$ and $b_k$ be respectively the payoffs of a C  and a D player when interacting with a group that consist of $k$ C players and $d-1-k$ D players. Then
\begin{equation}
\label{eq: payoff-no-incentive}
 a_k = \frac{(k+1)cr}{d} - c, \quad b_k=\frac{k r c}{d}.   
\end{equation}
Let $x$ be the frequency of cooperation. Using these expressions, \eqref{eq: equilibria2} become
\begin{align}
\label{eq: equilibria3}
 G_1(x)&=\frac{c}{d} \sum_{k=0}^{d-1}\begin{pmatrix}
d-1\\
k
\end{pmatrix}x^k (1-x)^{d-1-k}\Big[krq(1-2x)+(r-d)x(1-x-q)\Big]\notag
\\&=\frac{c}{d}\Bigg(rq(1-2x)\sum_{k=0}^{d-1}k\,\begin{pmatrix}
d-1\\
k
\end{pmatrix}x^k (1-x)^{d-1-k}+(r-d)x(1-x-q)\sum_{k=0}^{d-1}\begin{pmatrix}
d-1\\
k
\end{pmatrix}x^k (1-x)^{d-1-k}\Bigg)\notag
\\&=\frac{c}{d}\Bigg( rq(1-2x)\sum_{k=0}^{d-1}k\,\begin{pmatrix}
d-1\\
k
\end{pmatrix}x^k (1-x)^{d-1-k}+(r-d)x(1-x-q)\Bigg),
\end{align}
since 
\[
\sum_{k=0}^{d-1}\begin{pmatrix}
d-1\\
k
\end{pmatrix}x^k (1-x)^{d-1-k}=(x+(1-x))^{d-1}=1.
\]
In addition, using
\[
k\,\begin{pmatrix}
d-1\\
k
\end{pmatrix}=(d-1)\begin{pmatrix}
d-2\\
k-1
\end{pmatrix}
\]
We can simplify the first summation in \eqref{eq: equilibria3} as
\begin{align*}
\sum_{k=0}^{d-1}k\,\begin{pmatrix}
d-1\\
k
\end{pmatrix}x^k (1-x)^{d-1-k}&=(d-1)   \sum_{k=1}^{d-1}\begin{pmatrix}
d-2\\
k-1
\end{pmatrix}x^k (1-x)^{d-1-k} 
\\&=(d-1)x   \sum_{k=1}^{d-1}\begin{pmatrix}
d-2\\
k-1
\end{pmatrix}x^{k-1} (1-x)^{d-1-k} 
\\&=(d-1)x   \sum_{k=0}^{d-2}\begin{pmatrix}
d-2\\
k
\end{pmatrix}x^{k} (1-x)^{d-2-k}
\\&=(d-1)x(x+(1-x))^{d-2}
\\&=(d-1)x.
\end{align*}
Thus, \eqref{eq: equilibria3} reduces to
\begin{align*}
G_1(x)&= \frac{c}{d} x\Big(rq(d-1)(1-2x)+(r-d)(1-x-q)\Big)
\\&=\frac{c}{d} x \Big(rq(d-1)+(r-d)(1-q)-x(2rq(d-1)+(r-d))\Big).
\end{align*}
Therefore,
\begin{align*}
G(x)&=G_1(x)-\mu(2x-1)
\\&=\frac{c}{d}x \Big(rq(d-1)+(r-d)(1-q)-x(2rq(d-1)+(r-d))\Big)-\mu(2x-1)
\\&=-\frac{c}{d}(2rq(d-1)+(r-d)) x^2+ x \Big[\frac{c}{d}\big(rq(d-1)+(r-d)(1-q)\big)-2\mu\Big]+\mu.
\end{align*}
Thus, the mutator-replicator dynamics for PGGs without institutional incentives is given by
\begin{equation}
 \label{eq: RME no incentive}
\dot{x}=G(x)=-\frac{c}{d}(2rq(d-1)+(r-d)) x^2+ x \Big[\frac{c}{d}\big(rq(d-1)+(r-d)(1-q)\big)-2\mu\Big]+\mu.
\end{equation}
Equilibria of this dynamics are solutions in the interval $[0,1]$ of $G(x)=0$. In Section 3.1, we will analyse the solvability, stability and bifurcation analysis of this equation.
%In this section, we study the replicator-dynamics for PGGs with institutional incentives. We show the possible number of equilibria, their stability and  bifurcations analysis, as well as the asymptotic behaviour when the size of the group is large.
\subsection{PGGs with institutional Incentives}
%The equation \ref{eq: G} can be modified by implementing two additional factors: institutional incentive $\omega$ and incentive per capita $\delta$. This "introduces" institutions into the game, allowing to control and modify their dynamics. As we will see below, the difference between the behaviour of games with and without incentives is staggering. 

Now we consider PGGs with institutional incentives. Let $\delta>0$ be the per capita incentive. The total budget $d\delta$ is divided into two parts using a weight $0\leq \omega\leq 1$: the first part $\omega d \delta$ is equally rewarded to the cooperators increasing their payoff, while the other part $(1-\omega)d \delta$ is used for equally punished the defectors deducting their payoff. The rewarding and punishing are respectively leveraged by some  factors $a, b>0$.

Thus we obtain the following explicit formula for the payoff entries $a_k$ and $b_k$ for $k=0,\ldots, d-1$:
\begin{equation}
\label{eq: payoff-incentive}
 a_k = \frac{(k+1)cr}{d} - c+\frac{a \omega d \delta }{k+1}=a^0_k+\frac{a \omega d \delta }{k+1}, \quad b_k=\frac{k r c}{d}-\frac{b(1-\omega)d\delta}{d-k}=b^0_k-\frac{b(1-\omega)d\delta}{d-k},   
\end{equation}
where $a_k^0$ and $b_k^0$ are the corresponding payoffs when there is no incentive ($\delta=0$).

Clearly, $\delta=0$ corresponds to no-incentive studied in the previous section; pure reward and pure punishment correspond to $\omega=1$ and $\omega=0$ respectively.

Using \eqref{eq: payoff-incentive} we compute the expression inside the square bracket of the polynomial $G_1$ given in \eqref{eq: equilibria2}:
\begin{align*}
a_k x(1-x-q)-b_k (1-x) (x-q)&= \Big(a_k^0+\frac{a \omega d \delta }{k+1}\Big) x (1-x-q)-\Big(b_k^0-\frac{b(1-\omega)d\delta}{d-k}\Big) (1-x)(x-q)   
\end{align*}
Therefore $G_1(x)=G_1^0(x)+G_2(x)$, where 
\begin{align*}
G_1^0(x)&= \frac{c}{d} x\Big(rq(d-1)(1-2x)+(r-d)(1-x-q)\Big)
\\&=\frac{c}{d} x \Big(rq(d-1)+(r-d)(1-q)-x(2rq(d-1)+(r-d))\Big)
\end{align*}
(note that this is exactly is $G_1(x)$ in the case of no incentive), and
\begin{align*}
G_2(x)&=\sum_{k=0}^{d-1}\begin{pmatrix}
d-1\\
k
\end{pmatrix}x^k (1-x)^{d-1-k}\Big[\frac{a \omega d \delta }{k+1} x (1-x-q)+\frac{b(1-\omega)d\delta}{d-k}(1-x)(x-q)\Big]    
\\&:= G_3(x)+G_4(x),
\end{align*}
where
\begin{align*}
 G_3(x)&=\sum_{k=0}^{d-1}\begin{pmatrix}
d-1\\
k
\end{pmatrix}x^k (1-x)^{d-1-k}\frac{a \omega d \delta }{k+1} x (1-x-q),
\\ G_4(x)&=\sum_{k=0}^{d-1}\begin{pmatrix}
d-1\\
k
\end{pmatrix}x^k (1-x)^{d-1-k}\frac{b(1-\omega)d\delta}{d-k}(1-x)(x-q).
\end{align*}
Using the elementary identity
\[
\begin{pmatrix}
d-1\\
k
\end{pmatrix}\frac{d}{k+1}=\begin{pmatrix}
    d\\k+1
\end{pmatrix}
\]
we can transform $G_3$ as
\begin{align*}
G_3(x)&=  a \omega  \delta (1-x-q)\sum_{k=0}^{d-1}\begin{pmatrix}
d\\
k+1
\end{pmatrix}x^{k+1} (1-x)^{d-1-k}
\\&= a \omega  \delta (1-x-q)\sum_{k=1}^{d}\begin{pmatrix}
d\\
k
\end{pmatrix}x^{k} (1-x)^{d-k}
\\&=a \omega  \delta (1-x-q) \Big(1-(1-x)^d\Big),
\end{align*}
since 
\[
\sum_{k=0}^{d}\begin{pmatrix}
d\\
k
\end{pmatrix}x^{k} (1-x)^{d-k}=(x+1-x)^d=1.
\]
Similarly, using
\[
\begin{pmatrix}
d-1\\
k
\end{pmatrix}\frac{d}{d-k}=\begin{pmatrix}
d\\
k
\end{pmatrix}
\]
we have
\begin{align*}
G_4(x)&=b(1-\omega)\delta (x-q) \sum_{k=0}^{d-1}\begin{pmatrix}
d\\
k
\end{pmatrix}x^k (1-x)^{d-k}
\\&=b(1-\omega)\delta (x-q)(1-x^d).
\end{align*}
Bringing all together we have
\begin{align}
\label{eq: Gincentive}
G(x)&=G_1(x)-\mu(2x-1)=G_1^0(x)+G_3(x)+G_4(x)-\mu(2x-1)\notag
\\&=\frac{c}{d} x \Big(rq(d-1)+(r-d)(1-q)-x(2rq(d-1)+(r-d))\Big)-\mu(2x-1)\notag
\\&\qquad+a \omega  \delta (1-x-q) \Big(1-(1-x)^d\Big)+b(1-\omega)\delta (x-q)(1-x^d).
\end{align}
Thus, the replicator-mutator dynamics for PGGs with institutional incentive is given by
\begin{align}
\label{eq: RME with incentive}
\dot{x}&=G(x)\notag
\\&=\frac{c}{d} x \Big(rq(d-1)+(r-d)(1-q)-x(2rq(d-1)+(r-d))\Big)-\mu(2x-1)\notag
\\&\qquad+a \omega  \delta (1-x-q) \Big(1-(1-x)^d\Big)+b(1-\omega)\delta (x-q)(1-x^d).
\end{align}
Again, equilibria of this dynamics are roots in $[0,1]$ of the polynomial $G$.

Note that mathematically, the parameters $a$ and $b$ are redundant; one can show by solving a system of linear equations that any values of the  coefficients in front of $(1-x-q)\left(1-(1-x)^d\right)$ and $(x-q)(1-x^d)$ can be achieved by varying $\delta$ and $\omega$ -- $a$ and $b$ can be safely assumed to be equal to 1.  However, in order to preserve the model, we will not make this assumption (unless otherwise stated). Accordingly, equilibria of the replicator-mutator dynamics for PGGs with institutional incentives are roots in $[0,1]$ of the following polynomial equation:
\begin{equation}
\label{eq:g(x)=0}
    \begin{split}
&G(x)=\frac{c}{d} x \Big(rq(d-1)+(r-d)(1-q)-x(2rq(d-1)+(r-d))\Big)\\&\qquad-\mu(2x-1)
+a \omega  \delta (1-x-q) \Big(1-(1-x)^d\Big)+b(1-\omega)\delta (x-q)(1-x^d)  =0. 
    \end{split}
\end{equation}
In Section \ref{sec: results incentives}, we analyse the solvability in $[0,1]$ of this equation (thus the existence of the equilibria of the replicator-mutator dynamics), provide the stability and bifurcation analysis, reveal the role of the institutional incentive and of the combined mutations, as well as the asymptotic behaviour when the group size is sufficiently large.  
\section{Results}
\label{sec: results}
In this section, we present the main results of the present paper and discuss their biological interpretations. We start with public goods games without institutional incentives then with incentives.  For each mode, we identify the possible number of (stable and unstable) equilibria, analyse their stability properties, and investigate the dependence of these equilibria on the model's parameters via bifurcation and asymptotic analysis. To focus on the main ideas, we will postpone all the technically detailed proofs to the appendix.
\subsection{PGGs without institutional incentives}
\label{eq: result no incentives}
We recall that the replicator-mutator dynamics for PGGs is given in \eqref{eq: RME no incentive} and its equilibria are roots in $[0,1]$ of the quadratic polynomial equation
\[
G(x):= A x^2+Bx+\mu=0,
\]
where 
\begin{equation}
 \label{eq: AB}   
A:=-\frac{c}{d}(2rq(d-1)+(r-d)), \quad B:=\Big[\frac{c}{d}\big(rq(d-1)+(r-d)(1-q)\big)-2\mu\Big].
\end{equation}
\begin{figure}
    \centering
\includegraphics[scale=.57]{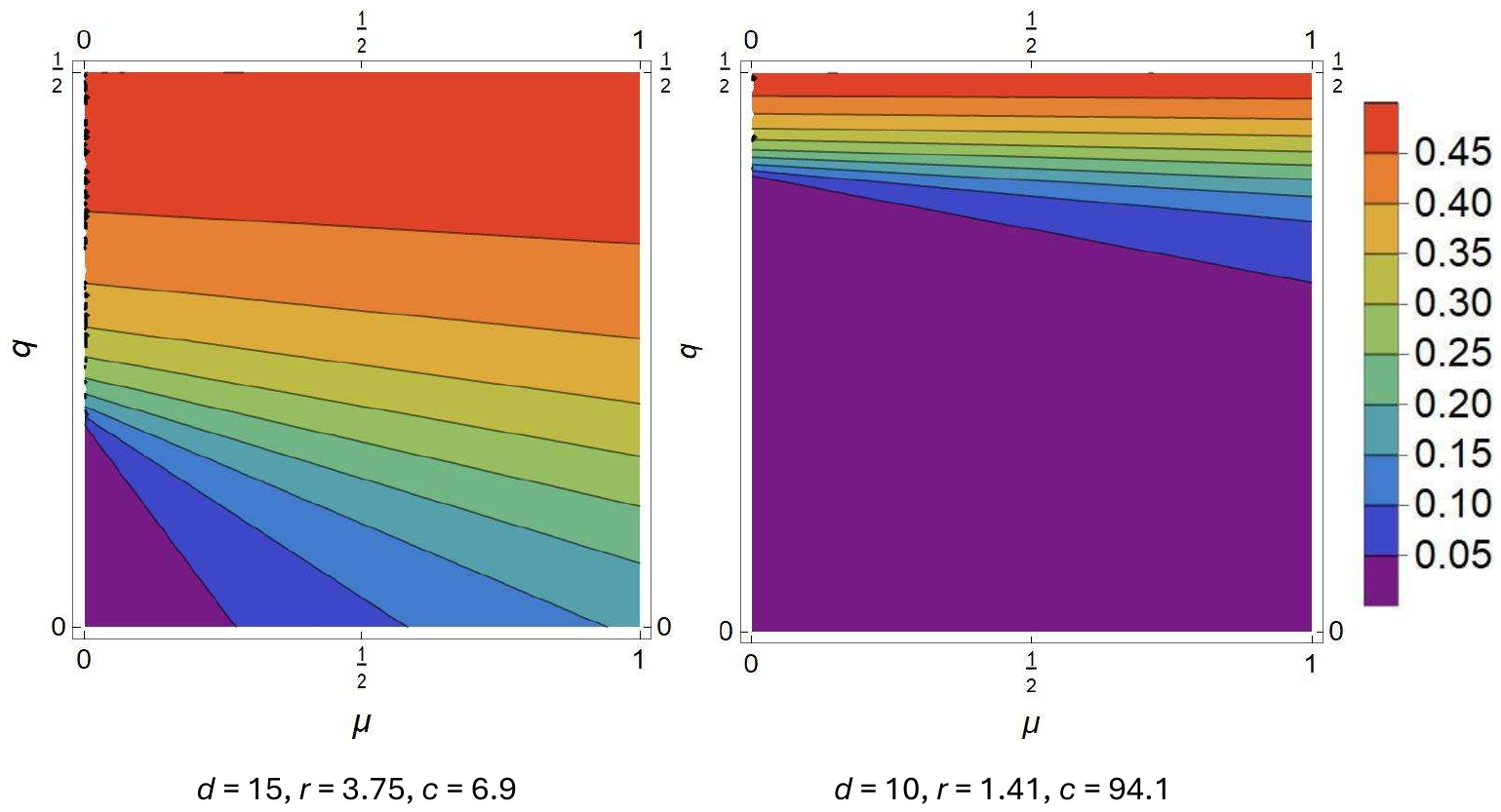}
    \caption{$x_1$ as a function of $\mu$ and $q$, for fixed values of $d, \ r$, and $c$.}
    \label{fig:value of x_1}
\end{figure}
\subsubsection{Equilibria}
Since $G$ is a quadratic polynomial, we can explicitly find its roots and determine whether a root belongs to the interval $[0,1]$ to become an equilibrium for the replicator-mutator dynamics. 
%We have 
%\begin{align*}
%&G(0)=\mu\geq 0, \quad G(1)=cq(1-r)-\mu<0;
%\\& G'(x)=2 Ax+B,\quad G'(0)=B=\big(rq(d-1)+(r-d)(1-q)-2\mu\big),
%\\& G'(1)=2A+B=-2(2rq(d-1)+(r-d))+\big(rq(d-1)+(r-d)(1-q)-2\mu\big)
%\end{align*}
%Recalling we assume that $1<r<d$ and that $G(x)$  is a linear or quadratic function; therefore, it can have only one nondegenerate equilibrium between 0 and 1. 
\begin{proposition}
\label{st: equilibria delta zero}
    The equation $G(x)=0$ has 
    \begin{enumerate}
        \item two equilibria at $x_0=0$ and $x_1=1$ when $\mu=q=0$;
        \item one equilibrium at $x_0=0$ and and an additional one at $x_1 = \frac{(d-2) q r+d (q-1)+r}{d (2 q r-1)-2 q r+r}$ if $q>\frac{d-r}{(d-2) r+d}$, when $q>0$, $\mu=0$. Additionally, $x_1<1/2$;
        \item one equilibrium at 
        \begin{small}
        \begin{equation}
        \label{eq: no incentive equilibrium}
            x_1=\frac{2 d \mu }{d \left(\sqrt{\frac{c^2 ((d-2) q r+d (q-1)+r)^2}{d^2}+4 c \mu  q (r-1)+4 \mu ^2}+2 \mu \right)-c ((d-2) q r+d (q-1)+r)}<\frac{1}{2}
        \end{equation}
        \end{small}
        when $\mu,q>0$. 
    \end{enumerate}
\end{proposition}
Note that the solution $x_1$ above is well-defined, since the term inside the square root is non-negative: the first summand is a square, as is the last; as for the middle one, we have $c>0, \mu>0,\ q>0$ and $r>1$. 
The values of $x_1$ as a function of $\mu$ and $q$,  for fixed $c,r,d$, are depicted in Figure \ref{fig:value of x_1} while numerical solutions of $x(t)$ are shown in Figure \ref{fig:solutionsnoincentiva}.
  \subsubsection{Stability and bifurcation analysis}
\subsubsection*{Stability analysis}
\begin{enumerate}
\item In absence of mutations, i.e. when $\mu=q=0$, the graph of the  function $G(x)$ is an upturned parabola; hence, $x_0$ is stable
while $x_1$ is unstable, since $G'(0)=(r-d)<0$ and $G'(1)=-(r-d) > 0$. Thus, in absence of mutations, the system dynamics always converges to full defection.
\item When  $\mu=0$,  $q>0$, $G_{\mu}(0)=0$  and $G_{\mu}(1)<0$ and the dynamics have one or two equilibria, as per Proposition \ref{st: equilibria delta zero}.

From the signs of $G_{\mu}(x)$ at the endpoints of the interval we can deduce the following:
\begin{itemize}
    \item If $q<\frac{d-r}{r(d-1)+(d-r)}$, $x=0$  is a unique, stable equilibrium.
    \item If $q>\frac{d-r}{r(d-1)+(d-r)}$, then $x=0$ is an unstable equilibrium while $x_1$ is a stable one.
\end{itemize}
\end{enumerate}
\begin{figure}
\centering
    \subfigure[$d = 7, c= 10, r = 5, q=0,\mu=0$]{\includegraphics[scale=.3]{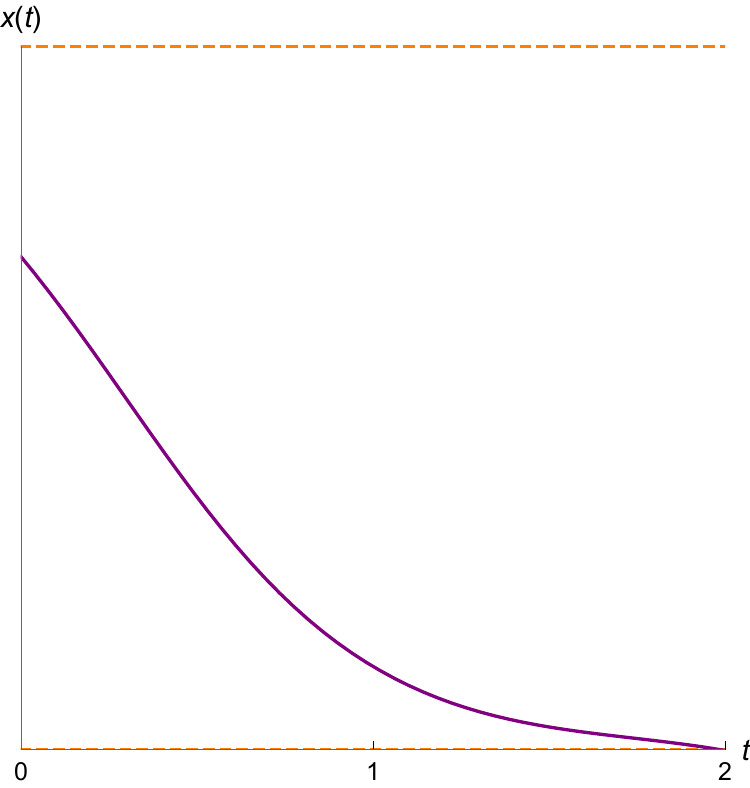}}
     \subfigure[$d = 7, c= 10, r = 5, q = 1/4, \mu=0$]{\includegraphics[scale=.345]{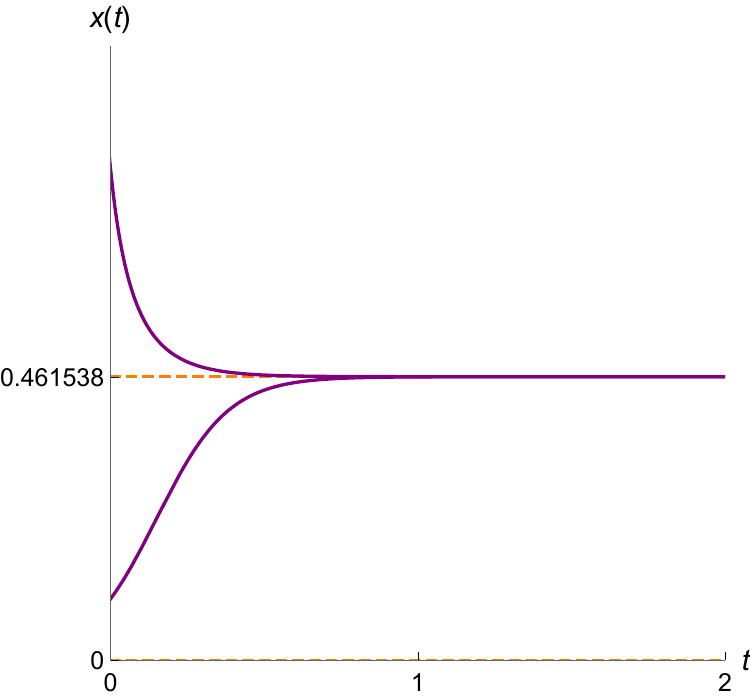}}
     \subfigure[$d = 7, c= 10, r = 5, q = 1/4, \mu=1/2$]{\includegraphics[scale=.345]{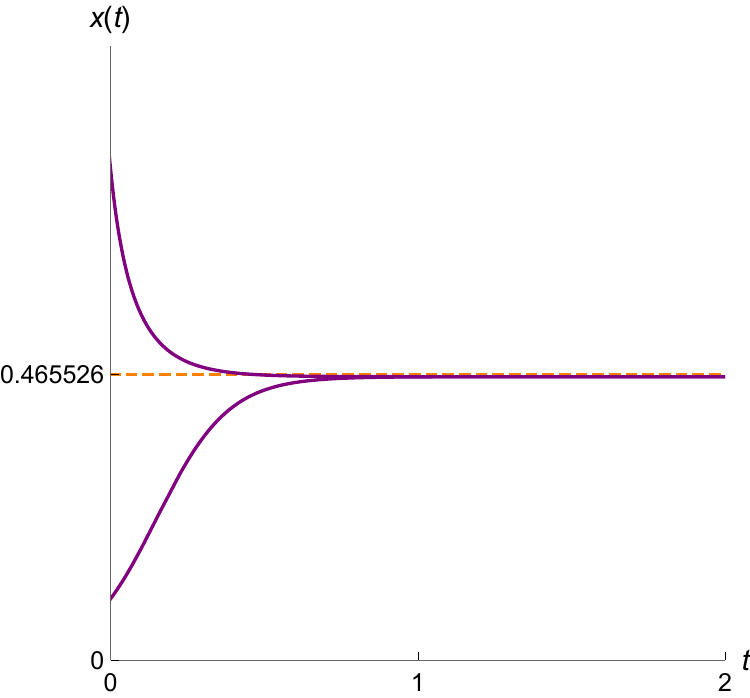}}
     \caption{Numerical solutions for PGGs without institutional incentives.}
     \label{fig:solutionsnoincentiva}
\end{figure}
\subsubsection*{Bifurcation analysis}
We study behaviour of the stable point with respect to the additive mutation $\mu$ when the multiplicative mutation $q$ is fixed.
\begin{proposition}
    \label{st: bifurcations degrre two}
\begin{enumerate}
    \item For a fixed (non-zero) value of $q$ the stable solution $x_1$ is an increasing function of $\mu$.
    \item For a fixed non-zero  $\mu$ the stable solution $x_1$ is an increasing function of $q$.
    \item $x_1$ has a limit as a function of $d$ 
    \[
    \lim\limits_{d\to\infty}x_1 = \frac{2 \mu }{\sqrt{c^2 (q r+q-1)^2+4 c \mu  q (r-1)+4 \mu ^2}-c (q r+q-1)+2 \mu }.
    \]
\end{enumerate}
\end{proposition}
\begin{figure}
    \centering
 \subfigure[$d = 3,q = 0.211,b = 1,a=3.12,\delta = 7, r= 2.88,\mu = 0.23,c= 2.8,\omega = 0.07$]{\includegraphics[scale=.5]{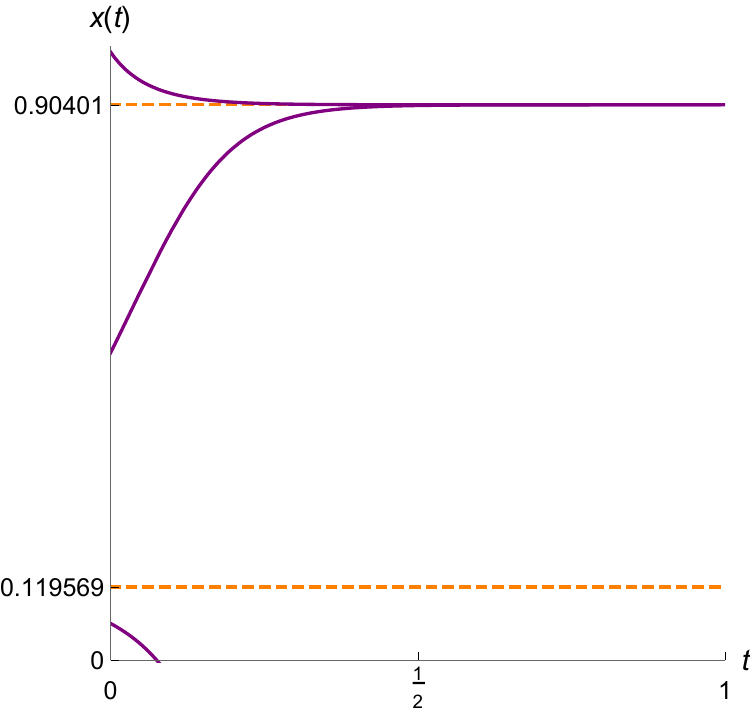}}
  \subfigure[$d=9,q= 0.211,b=0.34,a= 3.12,\delta = 6.7,r= 3.88,\mu = 0.23,c= 2.8,\omega = 0.07$]{\includegraphics[scale=.5]{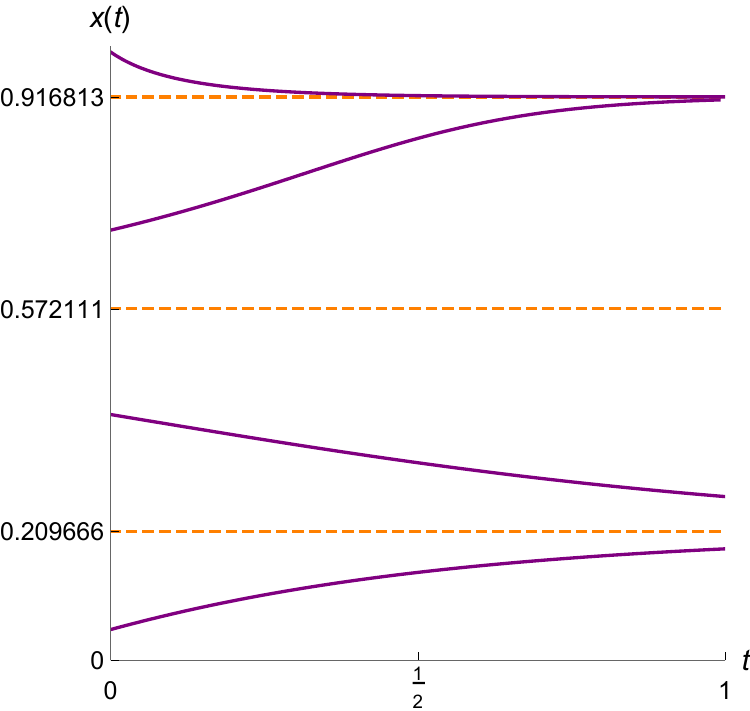}}
    \caption{Numerical solutions of $\dot{x}=G(x)$. The equilibria are given by dashed orange lines, and the non-constant solutions by purple lines.}
    \label{fig:graphs of soltutions}
\end{figure}
The above analysis demonstrates that,
in the absence of both mutations $\mu$ and $q$ (that is, for the replicator dynamics), the only equilibria are homogenous populations, with all players cooperating (unstable) or all players defecting (stable). On the other hand, even without institutional incentive, the replicator-mutator dynamics (that is, at least $q$ or $\mu$ is strictly positive) always has at least one stable equilibrium.

In the model where spontaneous mutations do not occur ($q=0$), the total cooperation strategy is an unstable equilibrium , while another equilibrium that favours defection is stable.  

When both types of mutation occur, the unique stable equilibrium still represents a population with more defectors than cooperators; however, increasing the likelihood of mutations brings the population closer to a balanced divide. 

\subsection{PGGs with institutional incentives}
\label{sec: results incentives}
\subsubsection{The number of equilibria }
We recall that the replicator-mutator dynamics for PGGs with institutional incentives is given in \eqref{eq: RME with incentive} and equilibria are roots in $[0,1]$ of the polynomial $G$, making the RHS of \eqref{eq: RME with incentive} vanishes.

In the following theorem, we show that $G$ has at most four roots, thus the replicator-mutator dynamics for PGGs with institutional incentives has at most four equilibria. This is an interesting result since $G$ is a polynomial of degree $d+1$, so in principle, it could have up to $d+1$ roots. Our proof makes use of the symmetric properties of $G$. 

For convenience, we sum up  the permitted values of the parameters in our model  below:
\begin{small}
\begin{equation}
\begin{split}
    \label{eq: table of parameters}
\delta>0, \ \ a>0, \ \ b>0,\ \  d\in\mathbb{Z}_+, d\ge 1, \ \ 1< r<d,\ \ \omega\in(0,1),\ \ q\in\left(0,\frac12\right), \ \ \mu\in(0,1),\ \ c>0
\end{split}
\end{equation}
\end{small}
\begin{theorem}
\label{st: general case}
    For any  values of the model parameters $\delta, a, b, d, r,\omega,q,\mu,c$  as in (\ref{eq: table of parameters}), the equation $G(x)=0$ has at most four solutions. 
\end{theorem}

All numbers of equilibria between $0$ and $4$ may occur. In Figure \ref{fig:number of solutions}, we plot  the graph of $G$ for some specific values of the parameters, each of the sub-figure demonstrates the possibility of having $0,\ldots, 4$ equilibria, and Figure \ref{fig:graphs of soltutions} shows some concrete numerical solutions.
 \begin{figure}
    \centering
    \subfigure[Zero solutions;$d=7, b=17.64, a=6.18,\delta=0.48, r=1.16,\omega=0.465,\mu=0.415,q=0.4505,c=20$]{\includegraphics[scale=.45]{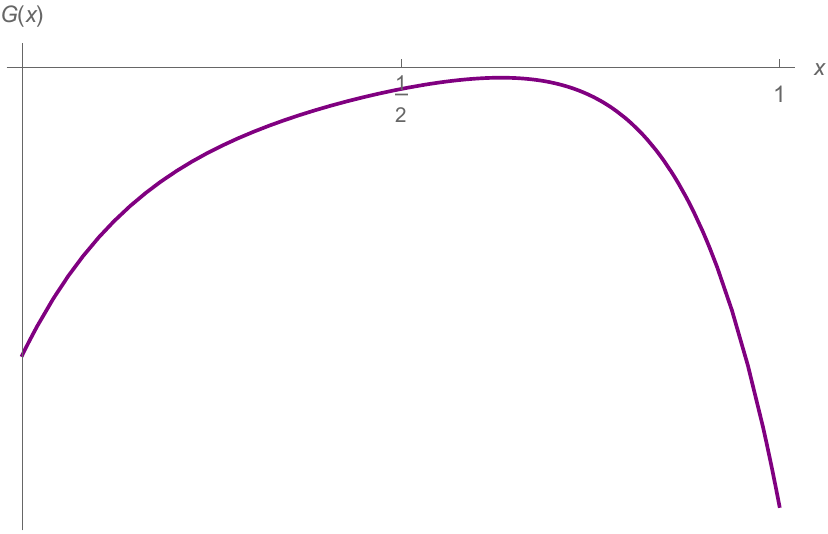}}
    \subfigure[One solution; $d=6, b=0.0001, a=15.32,\delta=1, r=5.57,\omega=0.133,\mu=0.643,q=0.3335,c=1$]{\includegraphics[scale=.45]{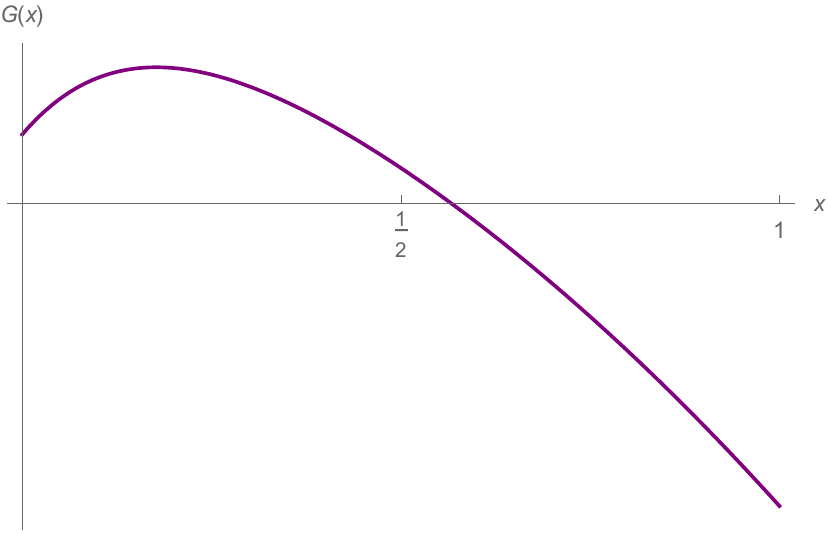}}
    \subfigure[Two solutions; $d=5, b=7.08, a=3.04,\delta=16.4, r=3.8,\omega=0.133,\mu=0.195,q=0.171,c=212$]{\includegraphics[scale=.45]{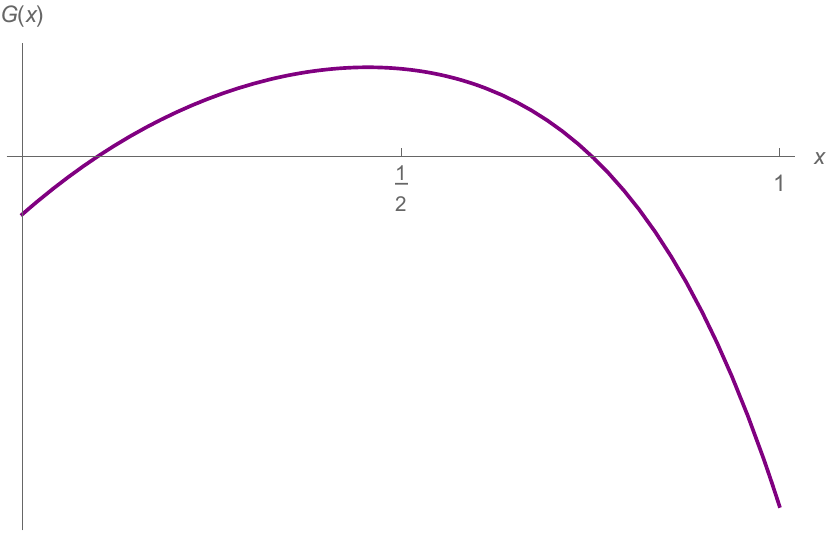}}
    \subfigure[Three solutions; $d=6, b=0.2, a=1.1,\delta=18.68, r=1.18,\omega=0.039,\mu=0.614,q=0.08,c=11.38$]{\includegraphics[scale=.45]{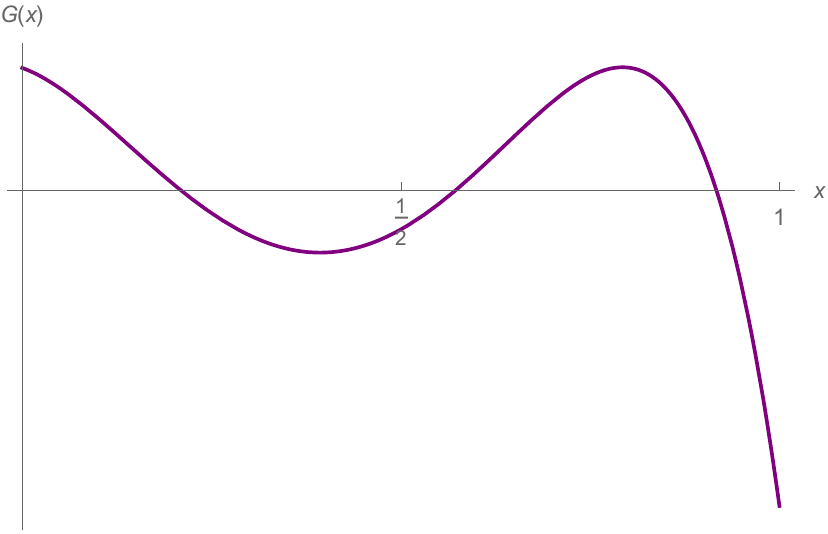}}
    \subfigure[Four solutions; $d=11, b=3.66, a=6.04\delta=16.4, r=1.05,\omega=0.133,\mu=0.643,q=0.195,c=174$]{\includegraphics[scale=.45]{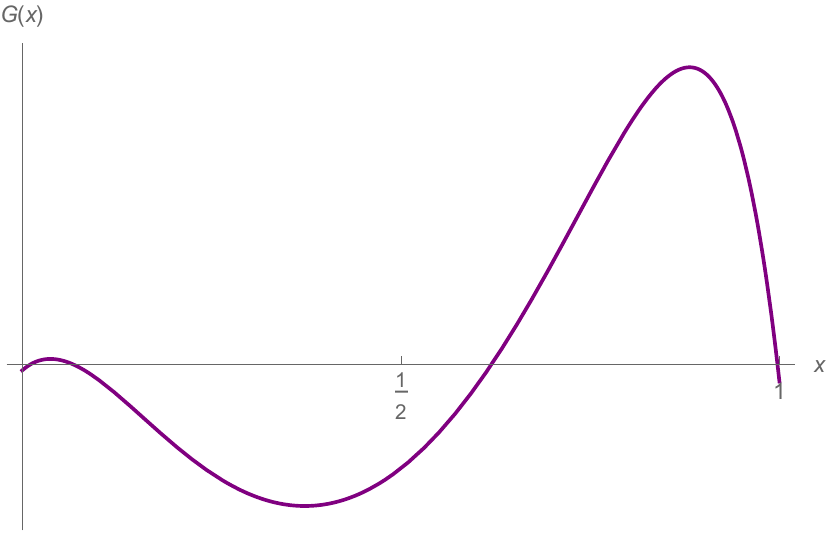}}
    \caption{Different possible numbers of solutions of the equation $G(x)=0$, in the presence of institutional incentives.}
    \label{fig:number of solutions}
\end{figure}

Furthermore, in Figure \ref{fig: random choice} we statistically and numerically calculate the probability of having a certain number of equilibria by randomly sampling the parameters. More precisely, we fix the value of $d$ for each simulation and impose the following restrictions on the rest of the parameters:
 \begin{equation*}
     \begin{split}
&\omega\in(0,1),\ \mu\in(0,1),\ \delta\in (0,10), \ c\in(0,5), r\in(1,d), \ q\in(0,1/2).
     \end{split}
 \end{equation*}
 For each iteration, we uniformly generate the parameters in the above intervals and solve the equation $G(x)=0$ using Wolfram Mathematica. The final output was the total number of cases with $0,1,2,3$ and $4$ solutions for the specified number of iterations. The graphic representation of the results is given in Figure \ref{fig: random choice}.

\begin{figure}[ht!]
 \centering
 \includegraphics[scale=.65]{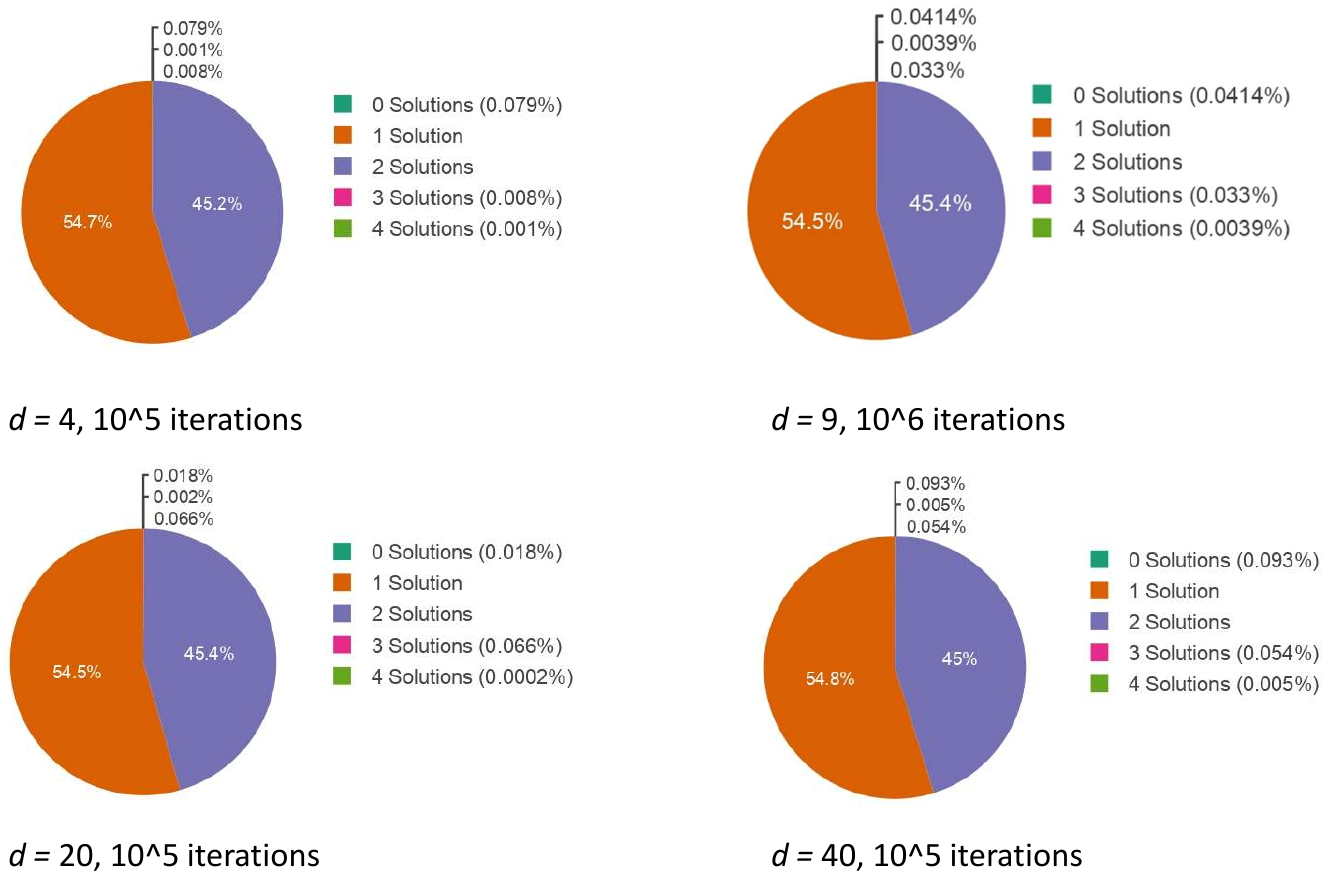}
\caption{Number of solutions for randomly chosen parameters, different values of $d$ and numbers of iterations}
\label{fig: random choice}
\end{figure}
We observe that the configurations with one or two solutions comprise the overwhelming majority of the cases. The ration between the two categories seems to be universal for all values of $d$; however, it can be checked that it changes when the permitted values of $c$ and $\delta$ are different. 
 
\subsubsection{Stability analysis}

 Substituting the endpoints of $[0,1]$ into $G(x)$, we get $G(0) = -b \delta  q (1-\omega)+\mu $, which clearly can have varying signs, but $ G(1) = -q \left(a \delta  \omega +c (r-1)\right)-\mu<0$. This entails that the following holds for simple equilibria of $G(x)=0$:
    \begin{enumerate}
        \item If (\ref{eq:g(x)=0}) has one equilibrium, it is stable;
        \item  If (\ref{eq:g(x)=0}) has two equilibria, they are, in order of increase, unstable and stable;
        \item If (\ref{eq:g(x)=0}) has three equilibria, they are, in order of increase, stable, unstable and stable;
            \item If (\ref{eq:g(x)=0}) has four equilibria, they are, in order of increase, unstable, stable, unstable and stable;
    \end{enumerate}
Thus, it follows that the largest equilibrium is always stable. Thus, if the cooperation is initially abundant, it will sustain around the largest equilibrium. In the case that the system has one or three equilibria, the smallest one is also stable; thus even if the cooperation is small initially, it will increase, reaching the smallest stable equilibrium.  On the other hand, in the case that the system has two or four equilibria, the smallest one will be unstable while the second one is stable. In these cases, if the level of cooperation is sufficiently large initially, more precisely if it is larger than the smallest unstable equilibrium, it will sustain and reach the second stable equilibrium. However, if the cooperation is too small initially (smaller than the smallest equilibrium), then it will die out. We also note that a stable equilibrium represents the co-presence of the strategies.  Thus, Theorem \ref{st: general case} and the above stability analysis provide quantitative insights into the emergence of cooperative behaviour, as well as the biological/behavioural diversity of evolutionary systems.
\subsubsection{The role of institutional incentives}
The main result of this section is the following theorem.
\begin{theorem}
\label{prop: sufficient incentive}    
When $\delta$ is sufficiently large (see \eqref{eq: lower bound} for an explicit lower bound), then the following statements hold
\begin{enumerate}
    \item There is at least one stable equilibrium $x^*$, satisfying $\frac{1}{2}<x^*<1$.
    \item $x^*$ increases with respect to $\delta$.
    \item $x^*$ decreases with respect to $\mu$.
    \item $x^*$ decreases with respect to $q$.
    \item $x^*$ decreases with respect to $\omega$, and satisfies $0.5<x_R^*<x^*_{mixed}<x^*_P<1$, where $x_R^*, x_{mixed}^*$ and $x^*_P$ respectively denote the stable equilibrium when using pure reward, mixed reward and punishment and pure punishment incentives.
\end{enumerate}
\end{theorem}
In the case of no incentive, we have shown in Section  \ref{sec: no institutional incentive} that the replicator-mutator dynamics has one stable equilibrium, in which the level  of cooperation  is less than $0.5$. This is also the case when the institutional incentive, $\delta$, is sufficiently small (see Appendix for a proof). 

Parts 1 and 2 of Theorem \ref{prop: sufficient incentive} clearly show the role of the institutional incentive in enhancing the level of cooperation: when the (per capita) budget for providing institutional incentives is sufficiently large, there is one stable equilibrium in which the level of cooperation is larger than $50\%$ (i.e. cooperation becomes more frequent than defection in the population). Moreover,  as this  budget increases,   the  cooperation also becomes more frequent.

Parts 3 and 4 reveal the effects of the additive and multiplicative mutations: in contrast to the per capita incentive, stable equilibria decrease with respect to both mutations. 

The last part, part 5, provides useful insight for the institution: punishment is more efficient than reward on enhancing the level of cooperation. 
%This is in accordance with previous works, in particular it extends \cite{dong2019competitive} to the case of positive multiplicative mutation.  

For an illustration of Theorem \ref{prop: sufficient incentive}, see Figure \ref{fig: sufficient incentive}.
\begin{figure}[ht!]
\centering
    \subfigure[$x^{\ast}$ as a function of $\mu$; $d=9, b=0.08,a=5.74,\delta=60.4, r=1.18,\omega=0.232,q=0.024,c=2.55$]{\includegraphics[scale=.33]{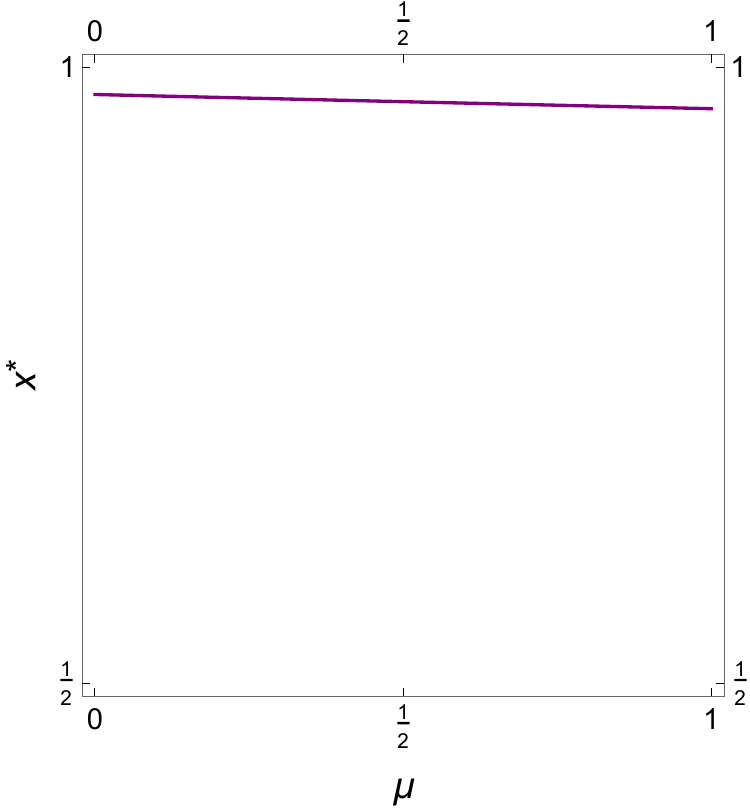}}
\subfigure[$x^{\ast}$as a function of $\delta$; $d=10, b=5.96, a=5,\omega = 0.476, r=8.39,\mu =0.146, q=0.464, c=8.75$ ]{\includegraphics[scale=.33]{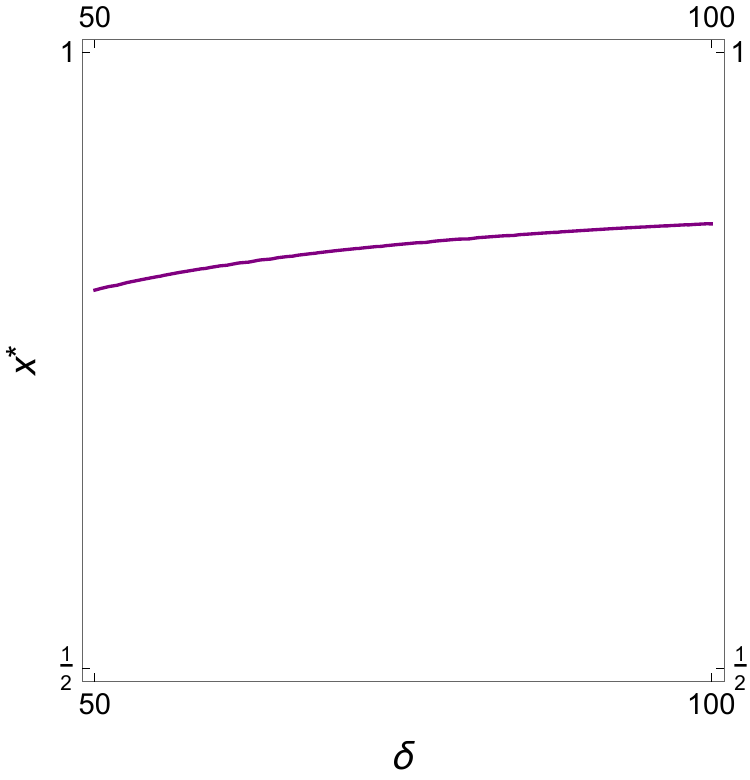}}
    \subfigure[$x^{\ast}$as a function of $\omega$; $d=8, b=4.02, a=8.62,\delta=181, r=1.93,\mu=0.138, q=0.116,c=12.15$]{\includegraphics[scale=.33]{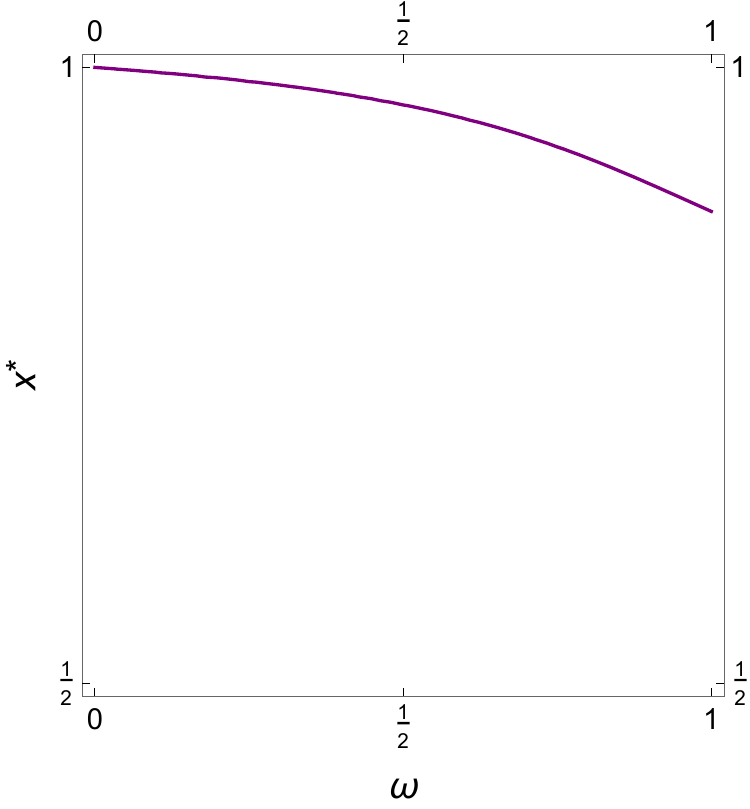}}
    \caption{Examples on the dependence of $x^{\ast}$on $\mu, \ \delta$ and $\omega$ for large values of $\delta$ (Theorem \ref{prop: sufficient incentive}).}
    \label{fig: sufficient incentive}
\end{figure}

\subsubsection{Special cases}
In Section \ref{sec: no institutional incentive} we have studied the case $\delta=0$; below we investigate what happens when other parameters in $G(x)$ assume values on the end points of the permitted intervals. The number in each entry of the table is the possible number of root of $G$. We recall the biological meaning of the parameters:
\begin{itemize}
    \item  $\delta>0$ is the per capita incentive, $\delta=0$ means no incentive is used.
    \item $\omega\in[0,1]$ represents the type of institutional incentive, $\omega=0$ (respectively, $\omega=1$) corresponds to pure punishment (resp. reward) while $0<\omega<1$ corresponds to a mixture of both reward and punishment.
    \item $q\in[0,1/2]$ and $\mu\in[0,1]$ model, respectively, the multiplicative and additive mutations. If one value is $0$, it means that the corresponding type of mutation is not considered.
\end{itemize}

\begin{theorem}
\label{st: particular cases}
Equation \eqref{eq:g(x)=0} has the following number of roots in the different cases represented in the table. 
 \begin{small}
    \[
    \begin{array}{c|c|c|c|c|c|c|c}
      &\delta = 0&\omega=0&\omega=1^{\ast}&\mu=0&\mu=1&q=0\\
      \hline
      \delta=0& \ 1&\ 1 & \ 1 & 1~ \text{or}~ 2 &1~ \text{or}~2 &\ 2&\\
      \hline
      \omega=0& \ 1&  \ \text{0, 1, 2 or 3}&  \times& \text{0 or 2 }&\text{0 or 2}& \text{1 or 3}\\
      \hline
      \omega = 1^{\ast}&1& \times & 1& 1&1& 1\\
      \hline
      \mu=0&\text{1 or 2}& \text{0 or 2 }& \text{1}&\text{0, 2 or 4 } &\times & \text{2 or 4 }\\
      \hline
      \mu=1&\text{1 or 2}& \text{0 or 2 }& \text{1}&\times  & \text{0, 1, 2, 3, or 4 }&\text{1 or 3 } &\\
      \hline
      q=0& \text{2} & \text{1 or 3 }& \text{1}&\text{2 or 4} &\text{1 or 3} &\text{1 or 3 } &\\
      \hline
      & & & & & & &
    \end{array}
    \]
    \end{small}
  The diagonal entries are the cases when only one condition is imposed; all other elements have two conditions imposed simultaneously.  The crosses denote incompatible conditions, such as $\mu=0$ and $\mu=1$.  

  $\ast$ -- unless $\mu=b q \delta$.
\end{theorem}
This theorem further demonstrates the effect of the parameters on the number of equilibria. 
\subsubsection{Conditions for one or two solutions}
As shown in Figure \ref{fig: random choice}, at least when the parameters are in the given intervals, the probabilities of having three or four equilibria are much smaller than those of having one or two equilibria.
For this reason, in this section, we describe the conditions under which equation $G(x) = 0$ has one or two solutions. We also note that the conditions for the existence of three and four solutions are indeed much more involved to pin down (for example, we conjecture that there is no Sturm sequence that fits the polynomial $G(x)$ for all the values of the parameters). 
\begin{figure}
    \centering
   \subfigure[General form   ]{\includegraphics[scale=.5]{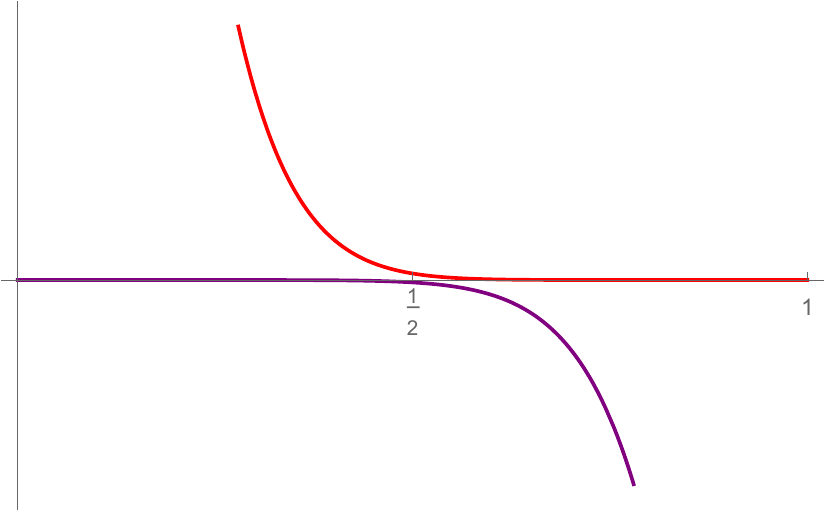}}
   \subfigure[Zoomed in]{\includegraphics[scale=.5]{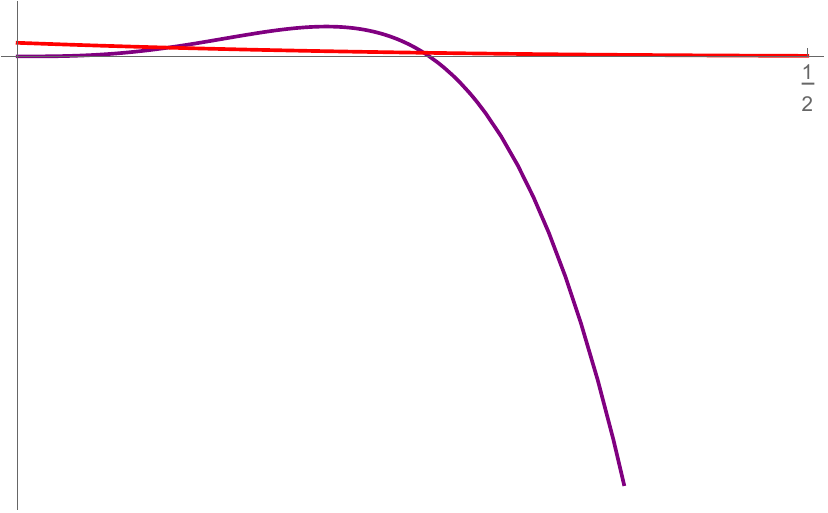}}
   \caption{Graphs of two parts of $G''_1(x)$}
    \label{fig:x and 1-x zoomed in}
\end{figure}
%\begin{remark}For simplicity, throughout this section we assume $a=b=1$ (see Remark \ref{rem: redundancy}).
%\end{remark}
%\vskip .1cm
Again, we rearrange the equation $G(x)=0$ as below
\begin{small}
\begin{equation}
\begin{split}
\label{eq: conditions for two solutions original}
&\underbrace{-\frac{c}{d} x \Big(rq(d-1)+(r-d)(1-q)-x(2rq(d-1)+(r-d))\Big)+\mu(2x-1)}_{G_2(x)}
\\&=\underbrace{\omega  \delta (1-x-q) \Big(1-(1-x)^d\Big)+(1-\omega)\delta (x-q)(1-x^d)}_{G_1(x)}. 
\end{split}
\end{equation}
\end{small}

We have remarked before that the graph of $G_2(x)$ is a parabola; here, we additionally observe that the graph of $G_1(x)$  is a downturned curve resembling a parabola. This gives rise to the following proposition.
\begin{proposition}
\label{st: when one or two solution}
  When the function $G_1(x)$ is convex and the graph of the function $G_2(x)$ is an upturned parabola, the equation (\ref{eq: conditions for two solutions original}) has one or two solutions.  In particular, it has one solution if $\mu
\geq\delta q(1-\omega)$ and two solutions if $\mu<\delta q(1-\omega)$.
\end{proposition}

By direct computations, we obtain the following:
\begin{equation}
\begin{split}
G_1(0) &= -\delta  q (1-\omega)<0,\\
G_2(0) &= -\mu<0, \\
G_1(1) &=-\delta  q \omega<0,\\
G_2(1) &= c q (r-1)+\mu >0
\end{split}
\end{equation}
Since the inequality $G_1(1)<G_2(1)$ always holds,  so if the conditions are as described in the statement of Proposition \ref{st: when one or two solution}, a solution will be unique.
On the other hand, when $\mu<\delta q (1-\omega)$, two solutions will exist.
\begin{figure}
    \centering
    \includegraphics[scale=.5]{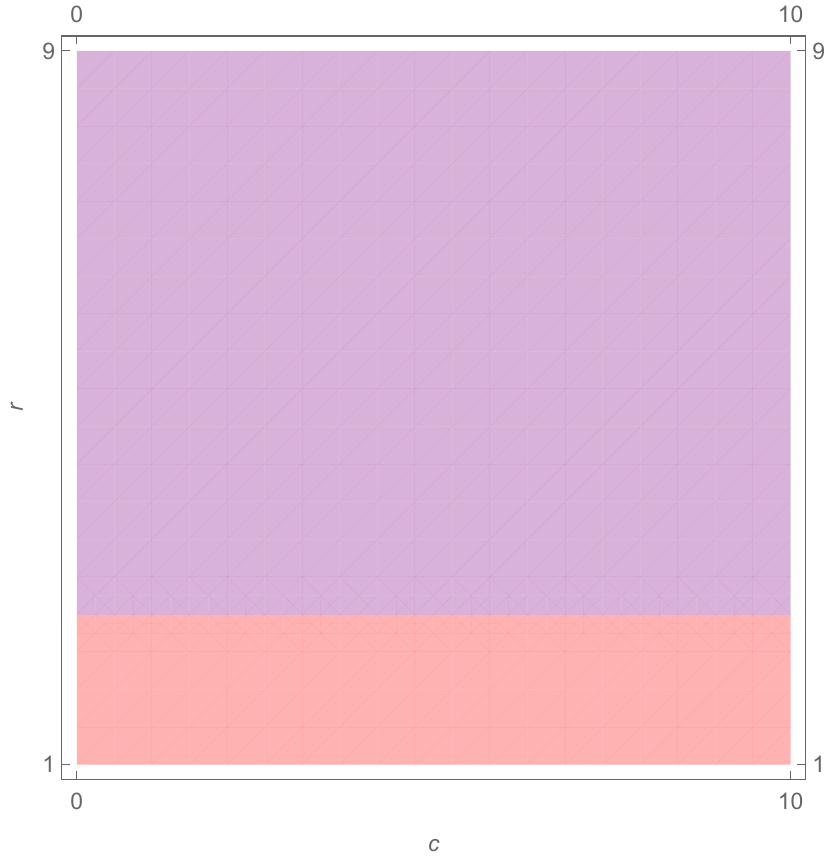}
    \caption{Conditions for no more than one or two solutions  in $c$ and $r$, $d=9, q = 0.1475$. The region in purple is the region corresponding to one or two solutions, the  region in red corresponds to possibly more solutions.  }
    \label{fig:conditions in c and r}
\end{figure}
In Section \ref{sec: conditions two solutions} of the appendix, we determine the conditions on $\omega$ and $q$ to satisfy Proposition \ref{st: when one or two solution}. More precisely, $q$ and $\omega$ respectively solve equations \eqref{eq: conditions on q} and \eqref{eq: conditions for omea}. By solving these equations, we can determine when the replicator-mutator dynamics has one or two solutions. We numerically plot this in Figure \ref{fig: one or two solutions in omaga and q}: the area with one or two solutions is coloured purple, and the area where more than two solutions may exist is coloured red.  
\begin{figure}[ht!]
    \centering
    \subfigure[Conditions in $\omega$]{\includegraphics[scale=.5]{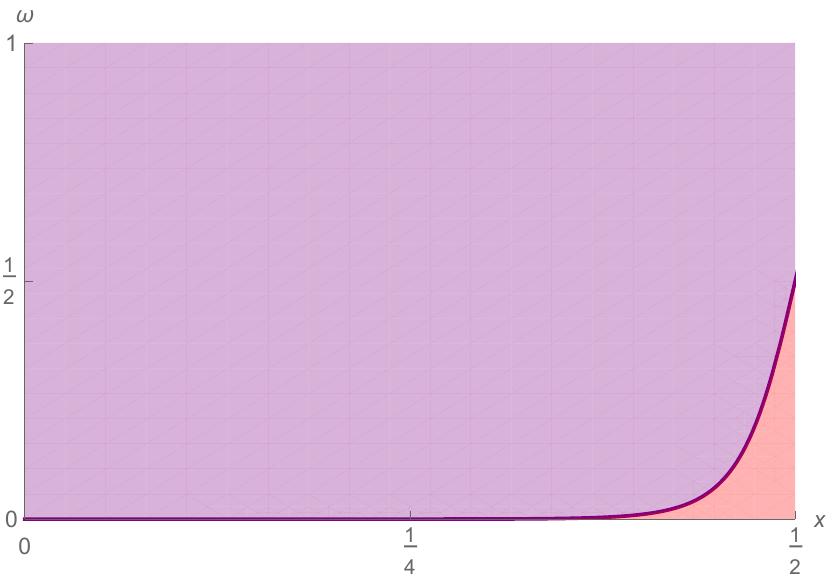}}
    \subfigure[Conditions in $q$]{\includegraphics[scale=.5]{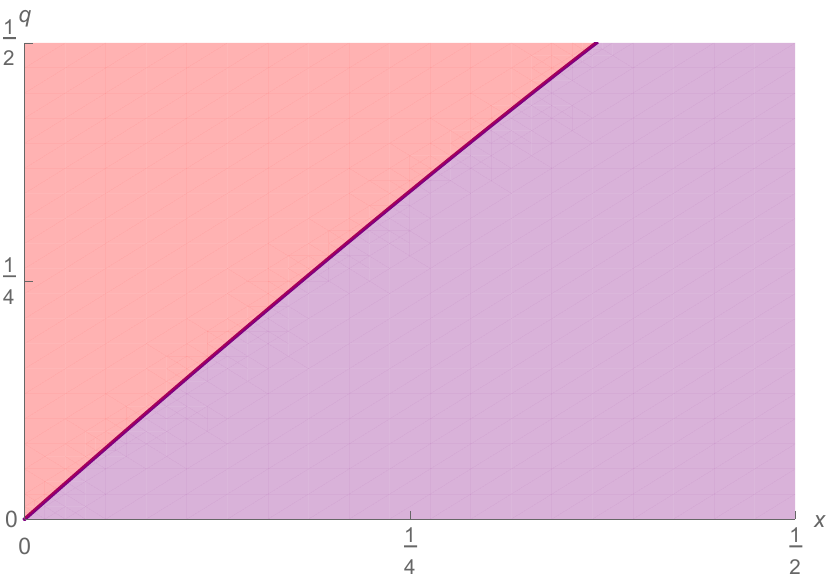}}
    \caption{Regions with one or two solutions (purple) and possibly more than two solutions (red) in $\omega$ and $q$ for $d=9$. }\label{fig: one or two solutions in omaga and q}
\end{figure}

\subsubsection{Bifurcations}

%In this section, we will discuss the effects that changing the values of the parameters has on the number and placement of solutions of $G(x)=0$. 

%\subsubsection{Classification of bifurcation diagrams in $\mu$}
In this section, we provide the detailed classification of the types of bifurcation diagrams in $\mu$. To construct bifurcation diagrams, we observe that $G(x)=0$ is linear in $\mu$ and therefore can be solved with respect to it and is rewritten as: 
    \begin{equation}
    \begin{split}
        \label{eq: mu bifurcations}
        G_{\mu}(x):=&\frac{d \delta  \left(a \omega  \left((1-x)^d-1\right) (q+x-1)-b (\omega -1) \left(x^d-1\right) (q-x)\right)}{d (2 x-1)} + \\
         &\frac{c x \left(d (x-1  + q(r - 2rx+1))+(2 q-1) r (x-1)\right)}{d (2 x-1)}=\mu
        \end{split}
    \end{equation}

Therefore, the number of solutions for a given fixed value $\mu = \mu_0$ is determined by the number of intersections of the graph of $G_{\mu}(x)$ with the horizontal line $y=\mu_0$.  

Immediately, we get the first geometric property from Theorem \ref{st: general case}:
\begin{lemma}
\label{lem: geo. property}
\begin{itemize} The following statements hold:
    \item No horizontal line intersects the graph of the function $G_{\mu}$ in more than 4 points.
    \item The graph $G_{\mu}(x)$ intersects the $x$-axis in 0,2 or 4  points.  
\end{itemize}
\end{lemma}
By construction, the graph of $G_{\mu}$ has a vertical asymptote at $x = 1/2$. Hence, it is completely determined by the signs of the denominator near $x=1/2$ and the minima and maxima of the function $G_{\mu}$. 

\begin{theorem}
\label{thm: bifurcation}
    Up to the  permitted relative vertical movement, the magnitude of the maxima and minima and horizontal cut-offs, there exist 19 types of bifurcation diagrams, as approximately drawn in Figures \ref{fig: bifurcations 1}, \ref{fig: bifurcations 2} and \ref{fig: bifurcations 3}. 
\end{theorem}
To make the classification manageable, we omit drawing certain types of diagrams that are not dissimilar enough from each other. Thus, we assume that the magnitude of the oscillations of maxima and minima can be arbitrary; in some cases, the graphs of the two parts can also be shifted vertically in relation to each other, as long as they retain the property of not intersecting any horizontal line in more than four places. Lastly, during initial computations we did not impose any restrictions on $\mu$; however, it is clear that the graph can be ``compressed" as much as necessary by choosing the appropriate values of $\delta$ and $c$ -- this allows us to assume that we can ``cut off" the graphs vertically at any value of $\mu$. 

As it can be  seen that the number of equilibria changes if one of the following two things occurs: i) a saddle-node bifurcation, with two equilibria merging into a degenerate one (or emerge from a degenerate one), or ii) appearance of a new equilibrium when the graph of $G_{\mu}$ crosses the boundary of $[0,1]\times[0,1]$-square where $x$ and $\mu$ lie.

We have grouped bifurcation diagrams in Figures \ref{fig: bifurcations 1}, \ref{fig: bifurcations 2} and \ref{fig: bifurcations 3}  by similarity; wherever the diagrams are similar, we put them in an increasing order with regard to the number of solutions in case $\mu=0$.
\begin{figure}
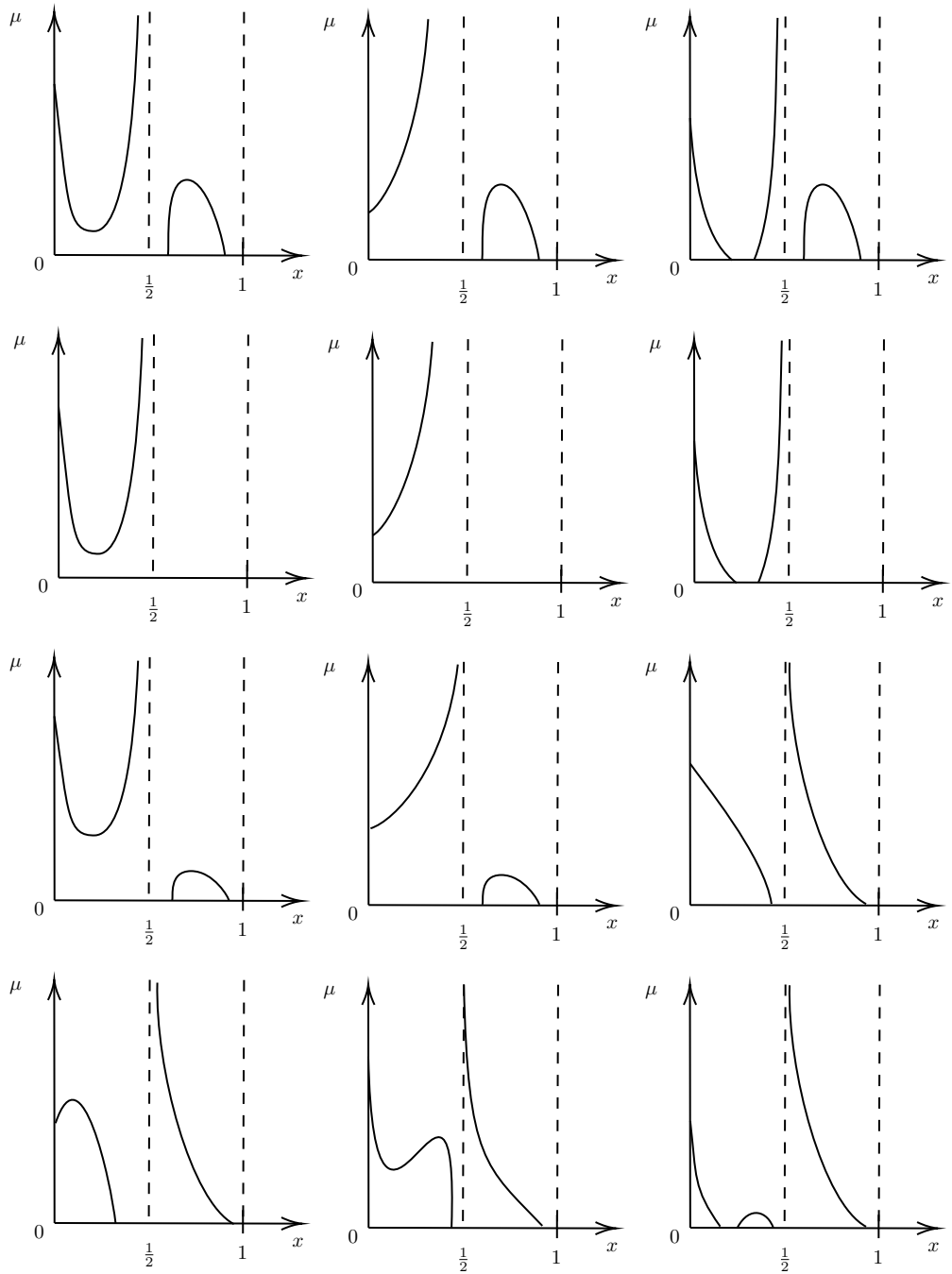

 \centering
\subfigure{
\tikzset{every picture/.style={line width=0.75pt}} %set default line width to 0.75pt        


}
\caption{Bifurcation diagrams}
    \label{fig: bifurcations 1}
\end{figure}
\begin{figure}
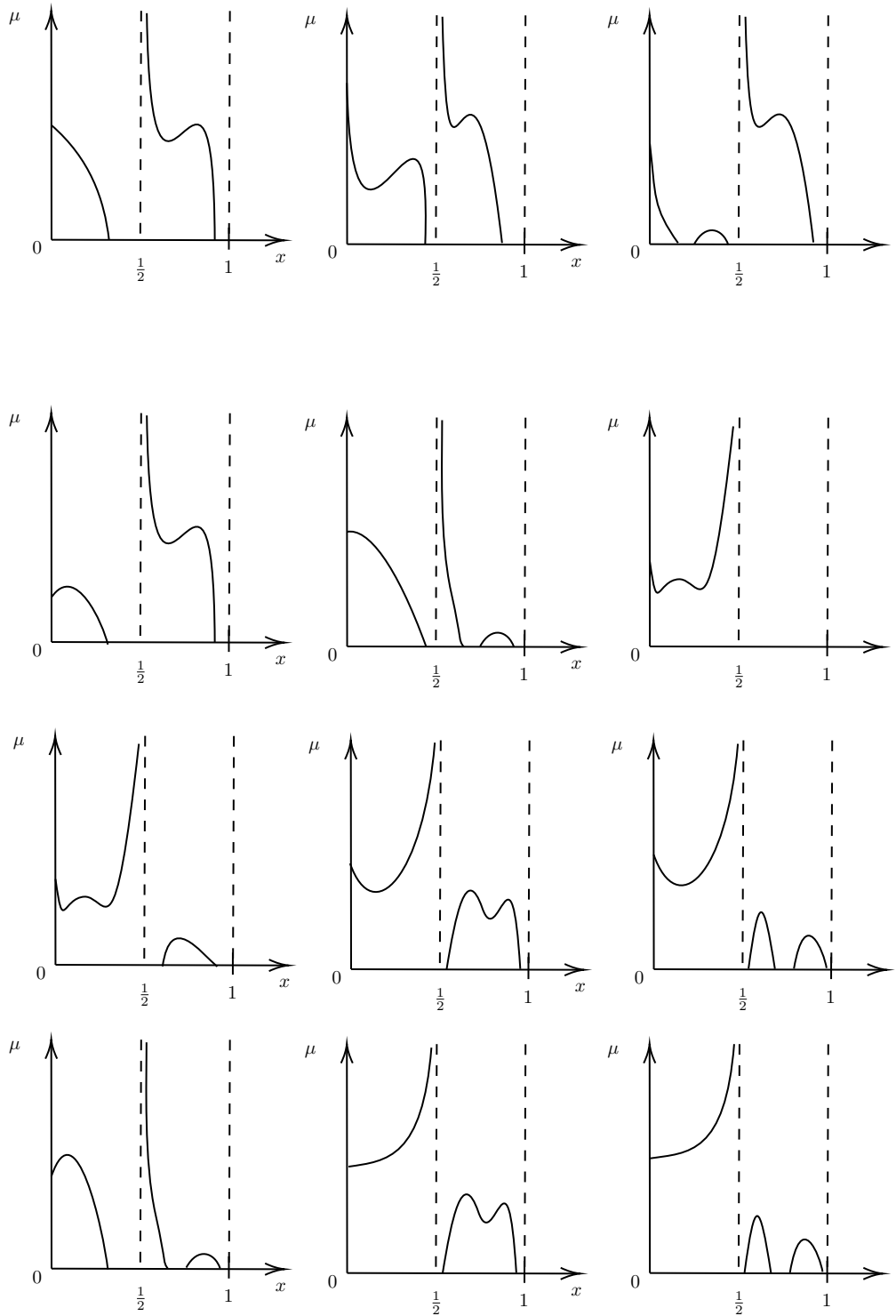

    \centering
  \subfigure{

\tikzset{every picture/.style={line width=0.75pt}} %set default line width to 0.75pt        

%
}
    \caption{Bifurcation diagrams, part 2 }
    \label{fig: bifurcations 2}
\end{figure}
\begin{figure}
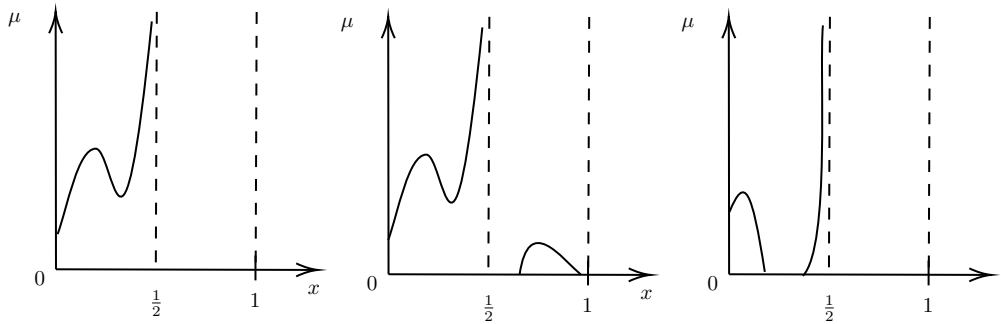

    \centering
 \subfigure{

\tikzset{every picture/.style={line width=0.75pt}} %set default line width to 0.75pt        

%
}
    \caption{Bifurcation diagram, part 3 }
    \label{fig: bifurcations 3}
\end{figure}
\subsubsection{Large group size}
The key technical challenge in the mathematical analysis to obtain results in the previous section is that $d$ is large but finite. Recall in particular that equilibria of the replicator-mutator dynamics are roots in $[0,1]$ of the polynomial $G$ given in \eqref{eq: RME with incentive}, which is of degree $d+1$. Solving $G$ is analytically impossible when $d$ is larger than $4$ according to the celebrated Abel's impossibility theorem. In this section, we investigate the equation $G=0$ in the limit when $d$ tends to infinity. It turns out that in this limit, $G$ and its roots can be approximated by a simpler function.

Note again that, $G(x)=0$ is equivalent  to
\[
G_1(x)=G_2(x),
\]
where
\begin{equation*}
\label{eq: right hand side d to infinity}
   G_1(x):= a \omega \delta \left(1-(1-x)^d\right) (1-x-q)+b\delta (1-\omega) \left(1-x^d\right) (x-q),
    \end{equation*}
and 
\[
G_2(x):= -\frac{c x (d q (r (2 x-1)-1)+d (1-x)+r (2 q (1-x)+x-1))}{d}+\mu  (2 x-1).
\]
We start by investigating the limiting  behaviour of $G_1$ as $d$ tends to infinity. Near $x=0$, this function increases as $a \omega \delta \left(1-(1-x)^d\right) (1-q)$ and  near $x=1$, decreases as $b \delta(1 - q) (1 - x^d)  (1 - \omega)$.  Roughly in the  middle of the $[0,1]$-interval, both  functions $ (1 - x^d) $ and $\left(1-(1-x)^d\right)$ are very close to 1, and $G_1$ is well-approximated by a linear function 
    
    \[
    g_1(x):=  x\delta\left(b (1-\omega)-a \omega \right)+\delta(a (1-q) \omega -b q (1-\omega )),
    \]
    
See Figure \ref{fig:dtoinfinity} for an illustration. 
\begin{figure}[ht!]
    \centering
    \includegraphics[width=0.5\linewidth]{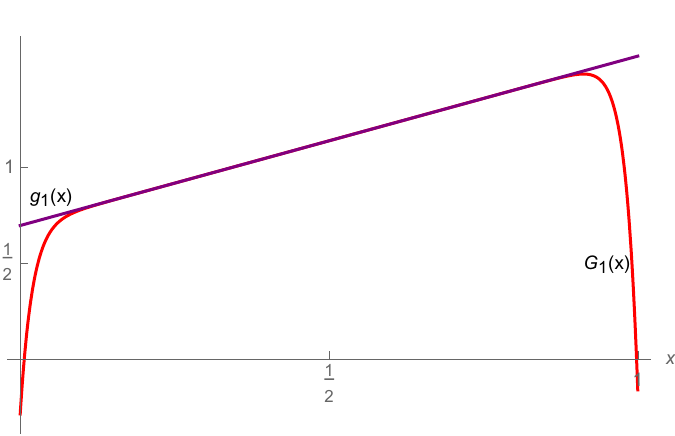}
    \caption{Approximation of the function $G_1(x)$ by a linear function $g_1(x)$ when $d\to\infty$ }
    \label{fig:dtoinfinity}
\end{figure}
For $G_2$, we have
\begin{equation}
\begin{split}
    \label{eq: the other side d to infinity}
    \lim\limits_{d\to\infty}G_2(x)&= x^2 c(2  q r-1)+x (c (1-q -qr)+2 \mu )-\mu=:g_2(x).
    \end{split}
\end{equation}
Whether the graph of this function is an upturned or downturned parabola depending on the parameters $c, q $ and $r$ -- observing that two of the three are not present in the left hand side of the equation.

Therefore, when $d\to\infty$, we effectively investigate the intersections of a line with a parabola, the number of which cannot exceed 2. Thus, it is  most likely that the limiting number of solutions of $G(x)=0$ will be equal to two. We have the following result.
\begin{lemma}
\label{lem: approximation}
With an appropriate choice of the parameters $\omega, a,b,\delta, r$ and $c$, the derivative of the approximating function $g_1(x)-g_2(x)$ will only have one zero and the value of the function in it will be large enough to prevent $G(x)$ from having any other zeros around it. 
\end{lemma}
\section{Summary and outlook}
\label{sec: summary}
Understanding the evolution and mechanisms for enhancing cooperation is a central question in evolutionary biology, economics and social sciences. The present paper pertains to this research direction. We have focused on the replicator-mutator dynamics, with combined additive and multiplicative mutations, for public goods games both in the absence or presence of institutional incentives. 
Through delicate and rigorous mathematical analysis, our results have offered new insight into the role of institutional incentives and of the effects of combined additive and multiplicative mutations on the evolution and promotion of cooperation in the context of public goods games. 
 
In this work, we have focused on models with one population and two strategies. However, in more complex applications, it is necessary to employ models with more than two strategies and/or more than one population. For instance, the replicator-mutator dynamics used in \cite{kauhanen2020replicator} to study language dynamics contains at least two populations. When consider more complex types of incentives such as peer ones \cite{sigmund2001reward,han2024evolutionary}, or when reciprocity \cite{ohtsuki2006leading,schmid2021unified} and commitment \cite{Han:2014tl} mechanisms are considered, the strategy space can be extremely large. As another instance,  the baseline models in recent works that employ evolutionary game theory to study AI governance consist of  three populations (of AI regulators, AI users and  AI developers) \cite{alalawi2024trustairegulationdiscerning} and four populations (of AI regulators, AI users,  AI developers, and commentariat) \cite{balabanova2025mediaresponsibleaigovernance}. In light of these important applications, in future works we will generalise this paper to multi-strategies and multi-populations evolutionary game models.
\section{Appendix}
In this appendix we present the detailed proofs of the results of the previous sections.
\subsection{Invariant space of the replicator-mutator dynamics}
\label{sec: invariant}
As already mentioned, the specific form of \eqref{eq: RME} ensures that the simplex
\[
S_n:=\{x=(x_1,\ldots, x_n)\in \mathbb{R}^n: \sum_{i=1}^nx_i=1\quad\text{and}\quad x_i\geq 0 \quad\text{for}\quad i=1,\ldots, n\}
\]
is invariant under \eqref{eq: RME} provided that initially it belongs to this set. In fact, let $s(t):=\sum_{i=1}^n x_i(t)$. Then
\begin{align*}
    \dot{s}&=\sum_{i=1}^n \dot{x}_i
    \\&\overset{\eqref{eq: RME}}{=}\sum_{i=1}^n\Big[\sum_{j=1}^n x_j f_j(\x)q_{ji}- x_i \f(\x)-\mu (n x_i-1)\Big]
    \\&=\sum_{j=1}^n x_j f_j(\x)\sum_{i=1}^n q_{ji}- (\f(\x)+\mu n)\sum_{i=1}^n x_i+ n\mu
    \\&\overset{\eqref{eq: sum of qij}}{=}\sum_{j=1}^n x_j f_j(\x)- (\f(\x)+\mu n)\sum_{i=1}^n x_i+ n\mu
    \\&=(1-s)(\f(\x)+n\mu)\quad (\text{since}~~\f(\x)=\sum_{j=1}^m x_j f_j(\x)).
\end{align*}
It implies that $s=1$ is a stationary point. Therefore, if $s(0)=1$ then $s(t)=1$ for all $t>0$. In other words, if the initial data satisfies $\sum_{i=1}^n x_i(0)=1$, then for any $t>0$, we also have $\sum_{i=1}^n x_i(t)=1$. This condition is biologically consistent since $
\{x_i\}$ model the frequencies of the strategies and they make up the whole population. 
%In addition, if $x_i$

\subsection{Proof of Proposition \ref{st: equilibria delta zero}}
  \begin{proof}[Proof of Proposition \ref{st: equilibria delta zero}]
        \begin{enumerate}
        \item Substituting $\mu=0$ and $q=0$ turns $G(x)=0$ into
        \[
        \frac{c (x-1) x (d-r)}{d} = 0,
        \]
        which clearly has roots at $0$ and $1$. 
        \item When only $\mu=0$, we get 
        \[
       G_{\mu}(x):= c x \left(\frac{(2 q-1) r (x-1)}{d}-2 q r x+q r+q+x-1\right) = 0.
        \]
        Once again, $x_0=0$ is an obvious solution, but so is 
        \[
        x_1 = \frac{(d-2) q r+d (q-1)+r}{d (2 q r-1)-2 q r+r}.
        \] By the conventions of our problem, this root must lie in $[0,1]$ interval; hence, we need to determine where the values of the parameters guarantee this outcome are.  Expressing $q$ through $r$ and $d$ yields
        \[
        q>\frac{d-r}{(d-2) r+d},
        \]
as in the statement of the theorem. 

To show that $x_1<\frac12$, we need to observe that $G_{\mu}\left(\frac12\right) = \frac{c (2 q-1) (d-r)}{4 d}<0$.

Thus, if the multiplicative mutation is sufficiently big, it enhances cooperation.
\item   The values of the function  $G(x)$ at $x=0$ and $x=1$ have opposite signs:
\begin{equation*}
    \begin{split}
        G(0) &= \mu>0\  \text{and}\\ 
        G(1)& = c q(1- r)-\mu<0 .
    \end{split}
\end{equation*}
        Since the function $G(x)$ is quadratic, the existence of a unique root immediately follows.

        Solving the equation explicitly gives two solutions:
\begin{small}
    \begin{equation}
    \begin{split}
        \label{eq: solutions quadratic}
        x_1&=\frac{2 d \mu }{d \left(\sqrt{\frac{c^2 ((d-2) q r+d (q-1)+r)^2}{d^2}+4 c \mu  q (r-1)+4 \mu ^2}+2 \mu \right)-c ((d-2) q r+d (q-1)+r)},\\
        x_2 &=  -\frac{2 d \mu }{d \left(\sqrt{\frac{c^2 ((d-2) q r+d (q-1)+r)^2}{d^2}+4 c \mu  q (r-1)+4 \mu ^2}-2 \mu \right)+c ((d-2) q r+d (q-1)+r)}
        \end{split}
    \end{equation}
\end{small}
However, under the imposed restrictions on the parameters the root $x_2$ does not belong to the $[0,1]$ (see Lemma \ref{st:x2 doesntbelong here} below).  Therefore, our unique root always coincides with $x_1$.

Lastly, 
\[
G\left(\frac12\right) = G_{\mu}\left(\frac12\right)<0,
\]
thus forcing $x_1<\frac12$.

        \end{enumerate}
    \end{proof}
    \begin{lemma}
        \label{st:x2 doesntbelong here}
        The solution $x_2$ from (\ref{eq: solutions quadratic}) does not belong in the $[0,1]$-interval.
    \end{lemma}
\begin{proof}[Proof of Proposition \ref{st:x2 doesntbelong here}]
Consider $x_2$ as a function of the parameter $q$:
\[
x_2(q) = -\frac{2 d \mu }{d \left(\sqrt{\frac{c^2 ((d-2) q r+d (q-1)+r)^2}{d^2}+4 c \mu  q (r-1)+4 \mu ^2}-2 \mu \right)+c ((d-2) q r+d (q-1)+r)}.
\]
This function has a singularity when $q = \frac{d-r}{r(d-2)}$ (a point that may belong to $[0,1/2]$); we set out to show that it is an increasing function of $q$.
\small
\[
x'_2(q) = \frac{2 c d \mu  \left(\frac{c ((d-2) r+d) ((d-2) q r+d (q-1)+r)+2 d^2 \mu  (r-1)}{d \sqrt{\frac{c^2 ((d-2) q r+d (q-1)+r)^2}{d^2}+4 c \mu  q (r-1)+4 \mu ^2}}+(d-2) r+d\right)}{\left(d \left(\sqrt{\frac{c^2 ((d-2) q r+d (q-1)+r)^2}{d^2}+4 c \mu  q (r-1)+4 \mu ^2}-2 \mu \right)+c ((d-2) q r+d (q-1)+r)\right)^2}.
\]
\normalsize
The denominator in this expression is a square and therefore positive, and so is the multiplier $2cd\mu$; therefore we only need to show that 
\[
\frac{c ((d-2) r+d) ((d-2) q r+d (q-1)+r)+2 d^2 \mu  (r-1)}{d \sqrt{\frac{c^2 ((d-2) q r+d (q-1)+r)^2}{d^2}+4 c \mu  q (r-1)+4 \mu ^2}}+(d-2) r+d>0.
\]
Through  transferring the fraction to the right hand side, multiplying by the (positive!) square root in the denominator and squaring both sides, we arrive at the inequality
\begin{equation*}
\begin{split}
&d^2 ((d-2) r+d)^2 \left(\frac{c^2 ((d-2) q r+d (q-1)+r)^2}{d^2}+4 c \mu  q (r-1)+4 \mu ^2\right)>\\&\left(c ((d-2) r+d) ((d-2) q r+d (q-1)+r)+2 d^2 \mu  (r-1)\right)^2,
\end{split}
\end{equation*}
which reduces to 
\[
4 d^2 \mu  (d-r) (c (r-1) ((d-2) r+d)+4 (d-1) \mu  r)>0,
\]
which is an always true statement. Note that this inequality yields positivity of the derivative  since $(d-2)r + d>0$.

Hence, since $x_2(q)$ has a singularity and is an increasing function, we only need to check that $x_2(0)>1$ and $x_2(1/2)<0$. We have 
\[
x_2(0)  = \frac{2 d \mu }{c (d-r)-\sqrt{c^2 (d-r)^2+4 d^2 \mu ^2}+2 d \mu }
\]
It is easy to see that $c (d-r)-\sqrt{c^2 (d-r)^2+4 d^2 \mu ^2}<0$ and therefore $c (d-r)-\sqrt{c^2 (d-r)^2+4 d^2 \mu ^2}+2 d \mu<2d\mu$, which immediately yields the desired result. 
Furthermore, 
\[
x_2(1/2) = \frac{2 \mu }{c(1- r)},
\]
clearly a negative expression. Therefore, $x_2(q)$ does not lie in the $[0,1]$-interval, which was what we set out to show. 
\end{proof}
\subsection{Proof of Proposition \ref{st: bifurcations degrre two}}
\begin{proof}[Proof of Proposition \ref{st: bifurcations degrre two}]
We prove points 1 and 2 in analogous ways, though applying  implicit function theorem:
\[
\frac{\partial x_1}{\partial\mu}=-\frac{\partial_{\mu} G}{\partial_x G}(x_1)=-\frac{1-2x_1}{\partial_x G(x_1)}>0
\]
since $1-2x_1>0, \partial_x G(x_1)<0$ (because $x_1$ is stable). Therefore, $x_1$ is increasing as a function of $\mu$.

Similarly,
\[
\frac{\partial x_1}{\partial q}=-\frac{\partial_q G}{\partial_x G}(x_1)>0
\]
since 
\[
\partial_q G=\frac{c}{d}x_1 (r(d-1)(1-2x_1)-(r-d))>0,
\]
from which it follows that $x_1$ increases as a function of $q$ as well.

To demonstrate point 3, we just need to consider the explicit form of $x_1$ from point 3 of  Proposition \ref{st: bifurcations degrre two}. Sending $d\to\infty$
gives an explicit expression:
\[
\lim\limits_{d\to\infty}x_1 = \frac{2 \mu }{\sqrt{c^2 (q r+q-1)^2+4 c \mu  q (r-1)+4 \mu ^2}-c (q r+q-1)+2 \mu }
\]
\end{proof}
\subsection{Proof of Theorem \ref{st: general case}}
\begin{proof}[Proof of Theorem \ref{st: general case}]
We employ Rolle's theorem in Calculus: if a derivative $F'(x)$ of the function $F(x)$, has $k$ zeros in some interval, the function $F(x)$ has fewer than $k+1$ zeros in that interval. 

Therefore, after demonstrating that in the general case $G'''(x)$ has exactly one zero in $[0,1]$, we will be able to claim that $G''(x)$ will have two or fewer zeros, $G'(x)$ three or fewer, and subsequently obtain our claim for $G(x)$. 

   Differentiating three times rids us of the quadratic part of the function:
    \begin{small}
    \begin{equation}
    \label{eq: G'''(x)}
    \begin{split}
        G'''(x) &=(d+1) d (d-1)  \delta  \left(-b  (1-\omega) x^{d-3} (x-q\frac{d-2}{d+1})+a \omega  (1-x)^{d-3} \left((1-x)-q\frac{d-2}{d+1} \right)\right).
        \end{split}
    \end{equation}
\end{small}
We are therefore looking for the roots of the equation 
\begin{equation*}
    b  (1-\omega) x^{d-3} \left(x-q\frac{d-2}{d+1}\right)=a \omega  (1-x)^{d-3} \left((1-x)-q\frac{d-2}{d+1} \right),
    \end{equation*}
which is equivalent to
\begin{equation}
    \label{eq: third derivative general case zeros}
     x^{d-3}\left(x - q\frac{d-2}{d+1}\right) = \frac{a \omega}{b(1-\omega)}(1-x)^{d-3}\left(1-x-q\frac{d-2}{d+1}\right).
    \end{equation}
If we denote $H(x):=x^{d-3}\left(x - q\frac{d-2}{d+1}\right)$, then the equation  (\ref{eq: third derivative general case zeros}) has the form 
\[
H(x) = K H(1-x)
\]
for some constant coefficient $K>0$. The graph of the function on the right-hand side is the vertically ($K$ times) stretched reflection respect to the line $x = \frac{1}{2}$ of the graph on the left-hand side.

     \begin{figure}[ht!]
        \centering
        \includegraphics[width=0.5\linewidth]{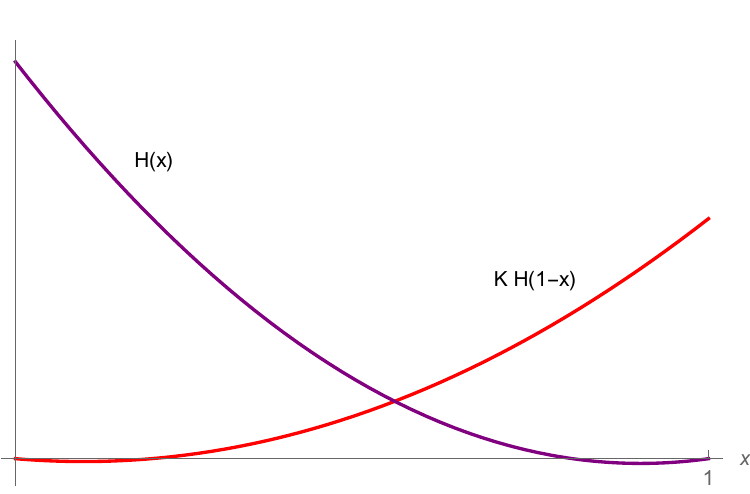}
        \caption{ The graphs of $H(x)$ and $KH(1-x)$.}
        \label{fig: third derivative}
    \end{figure}
Consider the left-hand side first, that is, $ x^{d-3}\left(x - q\frac{d-2}{d+1}\right)$. This function is negative when $0<x<q\frac{d-2}{d+1} $ and has a minimum at $x = q\frac{(d-3) }{d+1}$. Therefore, when $0<x<q\frac{d-3 }{d+1}$, the function is negative and decreasing, it is negative but increasing when $q\frac{d-3 }{d+1}<x<q\frac{d-2 }{d+1}$ and positive increasing when $x>q\frac{d-2}{d+1} $.

In contrast, the function  $(1-x)^{d-3}\left(1-x-q\frac{d-2}{d+1}\right)$ is positive and decreasing when $x<1-q\frac{d-2}{d+1}$, turns zero at $x = 1-q\frac{d-2}{d+1}$ , is negative and decreasing when $1-q\frac{d-2}{d+1}<x<1-q\frac{d-3}{d+1}$ and is negative but increasing when $1-q\frac{d-3}{d+1}<x<1$. Note that this can also be seen from treating the left-hand side as the reflection of the right-hand side, see Figure \ref{fig: third derivative} for an illustration.

Note that since $0\le q\le\frac12$, the points of interest are arranged in the following (strict) order:
\[
0<q\frac{d-3}{d+1}<q\frac{d-2}{d+1}<\frac{1}{2}<1 - \frac{d-2}{d+1}<1-q\frac{d-3}{d+1}<1.
\]

Therefore, the only interval where the intersections might take place is in the interval $q\frac{d-2}{d+1}<x<1 - \frac{d-2}{d+1}$, where both functions are positive (they do not have the same signs at any other intervals).

However, on that interval, the left hand side is strictly increasing, while the right hand side is strictly decreasing (multiplication by $K>0$ does not change that); this means that there may exist only one intersection. 
    \end{proof}

\subsection{Proof of Theorem \ref{prop: sufficient incentive}}
\begin{proof}[Proof of Theorem \ref{prop: sufficient incentive}]
We recall that equilibria of the replicator-mutator dynamics are solutions in the interval $[0,]$ of the following polynomial equation: 
\begin{align*}
G(x)&=\frac{c}{d} x \Big(rq(d-1)+(r-d)(1-q)-x(2rq(d-1)+(r-d))\Big)-\mu(2x-1)\notag
\\&\qquad+a \omega  \delta (1-x-q) \Big(1-(1-x)^d\Big)+b(1-\omega)\delta (x-q)(1-x^d).
\end{align*}
We have
\begin{align*}
&G(0)=\mu-b(1-\omega)\delta q, \quad G(1)=cq(1-r)-\mu-a\,\omega\,\delta\, q<0,
\\& G(1/2)=\frac{c}{2d}(r-d)(\frac{1}{2}-q)+ a\omega \delta (\frac{1}{2}-q)(1-(1/2)^d)+b (1-\omega)\delta (1/2-q)(1-(1/2)^d)  
\end{align*}
When
\begin{equation}
\label{eq: lower bound}
\delta>\max\Big\{\frac{\mu}{b(1-\omega)q}, \frac{c}{2d}\frac{d-r}{(a\omega+b(1-\omega))(1-(1/2)^d)}\Big\},    
\end{equation}
we have $G(0)<0, \quad G(1/2)>0, \quad G(1)<0$. Therefore, there is at least one stable equilibrium $x^*$, which is between $1/2$ and $1$. Thus it is larger than the stable equilibrium point in non-incentive dynamics. Next we will analyse the behaviour of $x^*$ as a function of each of the model's parameter. To this end, we calculate the derivative of $x^*$ with respect to the parameter using the implicit theorem.

%\begin{enumerate}
%    \item If $\mu=0, \omega(0)=1, q=0$ then $G(0)=G(1)=0$.
%    \item If $\mu=0, \omega(0)=1, q\neq 0$, then $G(0)=0, G(1)<0$.
%    \item $\mu\neq 0$ or $\omega(0)\neq 1$, then $G(0)>0$, $G(1)<0$. Then there exists at least one interior equilibrium $x^*\in(0,1)$. 
%\end{enumerate}
%\textcolor{red}{can we do more? Decartes's rule of signs, conditions to have a unique stable equilibrium, etc?}

%\textcolor{red}{Behaviour of $x^*$ as a function of $r, \mu, q,d$? analytical, plot, maximum values, etc}
%\begin{itemize}
%    \item  behaviour w.r.t. $q$ and $\mu$: effects of different types (or combined effect) of mutations
%    \item behaviour w.r.t. $\delta$: comparison between using incentive vs not using incentive 
%    \item behaviour w.r.t. $\omega$: comparison between different types of incentives (reward, punishment, mixed) \end{itemize}
\subsubsection*{Behaviour w.r.t. $\mu$}
Similarly as in the case of no institutional incentive, we calculate
\[
\frac{\partial x^*}{\partial\mu}=-\frac{\partial_{\mu} G}{\partial_x G}(x^*)=-\frac{1-2x^*}{\partial_x G(x^*)}<0
\]
since $x^*>1/2$ and $\partial_x G(x^*)<0$. Therefore, in constrast to the no-institutional incentive case, $x^*$ is decreasing as a function of $\mu$.
\subsubsection*{Behaviour w.r.t. $\delta$}
\[
\partial_\delta G(x^*, \delta)=a \omega  (1-x^*-q) \Big(1-(1-x^*)^d\Big)+b(1-\omega) (x^*-q)(1-(x^*)^d).
\]
It follows from the equation $G(x^*)=0$ that
\begin{align}
\partial_\delta G(x^*, \delta)&=a \omega  (1-(x^*)-q) \Big(1-(1-x^*)^d\Big)+b(1-\omega) (x^*-q)(1-(x^*)^d)\notag
\\&=-\frac{1}{\delta}\Bigg(\frac{c}{d} x^* \Big(rq(d-1)+(r-d)(1-q)-x^*(2rq(d-1)+(r-d))\Big)-\mu(2x^*-1)\Bigg).\label{eq: derivative delta}
\end{align}
Since $x^*>\frac{1}{2}$, we have
\begin{align*}
rq(d-1)+(r-d)(1-q)-x^*(2rq(d-1)+(r-d))&\leq rq(d-1)+(r-d)(1-q)-\frac{1}{2}(2rq(d-1)+(r-d))
\\&=(r-d)(\frac{1}{2}-q)\leq 0.
\end{align*}
Therefore, using $x^*>\frac{1}{2}$ again, 
\[
\frac{c}{d} x^* \Big(rq(d-1)+(r-d)(1-q)-x^*(2rq(d-1)+(r-d))\Big)-\mu(2x^*-1)< 0.
\]
From this and \eqref{eq: derivative delta}, we deduce 
\[
\partial_\delta G(x^*, \delta)>0.
\]
Hence, since $\partial_x G(x^*,\delta)<0$, we have
\[
\frac{\partial x^*}{\partial \delta}=-\frac{\partial_\delta G(x^*, \delta)}{\partial_x G(x^*,\delta)}>0
\]
Thus the stable equilibrium $x^*$ is increasing with respect to $\delta$. In other words, the more incentive is provided, the bigger the frequency of cooperation is. 

\subsubsection*{Behaviour w.r.t. $\omega$}
\[
\partial_\omega G= a\delta (1-x^*-q) \Big(1-(1-x^*)^d\Big)-b\delta (x^*-q)(1-(x^*)^d).
\]
Suppose $a=b$. Then
\begin{align*}
\partial_\omega G(x^*,\omega)&= \delta \Big[(1-x^*) \Big(1-(1-x)^d\Big)-x^*(1-(x^*)^d)\Big]+ q\delta \Big((1-x^*)^d-(x^*)^d)
\\&=\delta (1-x^*)x^* \sum_{i=0}^{d-1}\Big((1-x^*)^i-(x^*)^i\Big)+ q\delta \Big((1-x^*)^d-(x^*)^d)
\end{align*}
Since $x^*>\frac{1}{2}$, $(1-x^*)^i<(x^*)^i$ for all $i=0,\ldots, d$. This implies $\partial_\omega G(x^*,\omega)<0$. Thus

\[
\frac{\partial x^*}{\partial \omega}=-\frac{\partial_\omega G(x^*, \delta)}{\partial_x G(x^*,\omega)}<0
\]
Hence $x^*$ is decreasing with respect to $\omega$. Therefore, $0.5<x^*_R<x^*_{\text{mixed}}<x^*_{P}<1$, where $x^*_I, I\in\{R,\text{mixed}, P\}$ respectively denotes the stable equilibrium when using reward, mixed and punishment incentive. Thus, punishment is the most effective incentive in promoting cooperation. This generalizes the result of \cite{dong2019competitive} from $q=0$ to 
\subsubsection*{Behaviour w.r.t. $q$}
\[
\partial_q G(x^*,q)=\frac{c}{d}x^* (r(d-1)(1-2x^*)-(r-d))-a\omega \delta (1-(1-x^*)^d)-b(1-\omega)\delta (1-(x^*)^d).
\]
In general, this function can assume positive, ass well as negative values. Nonetheless, there are some observations that we can make.

The quadratic part of $\partial_qG(x)$ is
\[
c x \left(\frac{2 r (x-1)}{d}-2 r x+r+1\right),
\]
a function that assumes positive values when $x\in\left[0,\frac{(d-2) r+d}{2 (d-1) r}\right] $ (the latter need not be smaller than 1). 

On the other hand, $-a\omega \delta (1-(1-x^*)^d)-b(1-\omega)\delta (1-(x^*)^d)$ is a strictly negative expression. It has a unique minimum at  the point $x = \frac{\sqrt[d-1]{\frac{a \omega }{b (1-\omega )}}}{\sqrt[d-1]{\frac{a \omega }{b (1-\omega )}}+1}$ (obtained by differentiating and finding the unique zero of the derivative). 

Therefore, the largest value of $-a\omega \delta (1-(1-x^*)^d)-b(1-\omega)\delta (1-(x^*)^d)$ on $[0,1]$ is obtained either when $x=0$ or when $x=1$ -- these are, respectively, $-b \delta  (1-\omega )$ and $-a \delta  \omega$.

Therefore, if the parabolic part at its maximum assumes a value smaller than $\min\left(b \delta  (1-\omega ),a \delta  \omega\right)$, it will be guaranteed that the entire function is less than zero.  The maximum of the quadratic part is equal to $\frac{c ((d-2) r+d)^2}{8 (d-1) d r}$, which completes a part of the proof of the following 
\begin{lemma}
    If $\frac{c ((d-2) r+d)^2}{8 (d-1) d r}<\delta\min\left(b  (1-\omega ),a  \omega\right)$, then stable equilibria decrease while the unstable ones increase as functions of $q$. 
\end{lemma}

\begin{proof}
    This immediately follows from the discussion above: if the condition from the statement holds, $\partial_qG(x^{\ast}_q)<0$ for all $x^{\ast}$. Since $\partial_{x}G(x^{\ast},q)<0$ for $x^{\ast}$ stable and $\partial_{x}G(x^{\ast},q)>0$ for $x^{\ast}$ unstable, we get the appropriate increase and decrease. 
\end{proof}

\end{proof}

\begin{remark}[small institutional incentive]
By similar analysis, we can also quantitatively analyse the case when the institutional incentive is sufficiently small. More precisely, when $\delta$ is small,
\[
\delta<\min\Big\{\frac{\mu}{b(1-\omega)q},\frac{c}{2d}\frac{d-r}{(a\omega+b(1-\omega))(1-(1/2)^d)}\Big\}
\]
then $G(0)>0$, $G(1/2)<0$ $G(1)<0$. Then there is at least one stable equilibrium in $(0,1/2)$. Using the implicit theorem, we can also analyse the behaviour of the stable equilibrium with respect to the model's parameters as in the case of large institutional incentive.
\end{remark}

\subsection{Proof of Theorem \ref{st: particular cases}}
\begin{proof}[Proof of Theorem \ref{st: particular cases}]
      We go through the table case by case.\\
      \vskip .2cm
      \noindent$\boldsymbol{\delta = 0.}$ This case was addressed above in Section \ref{sec: no institutional incentive}, as well as the case $\delta=\mu=0$. 
 \vskip .2cm
      \noindent$\boldsymbol{\omega = 0.}$
Denote $G(x)$ with $\omega$ substituted for 0 by $F_{\omega_0}(x)$. Then 
\begin{equation}
    \label{eq: omega0}
    F_{\omega_0}(x) = -b\delta x^{d+1} + b\delta q x^d + Ax^2 + (B+b\delta)x + (\mu-b\delta q).
\end{equation}

We assume that $\mu - b\delta q\ne 0$.

Whatever the values of the parameters are, the first two coefficients will be negative; this justifies using Decartes' rule of signs. In order for the equation to have four roots (the maximal number), the coefficiens must have the following signs:
\begin{equation}
    \label{eq: signs omega0}
    \begin{array}{c c c c c}
        -b\delta & b\delta q &A &B+b\delta&\mu-b\delta q\\
        -&+&-&+&-
    \end{array}
\end{equation}
\begin{figure}
    \centering
  \subfigure[Zero solutions]{\includegraphics[scale=.5]{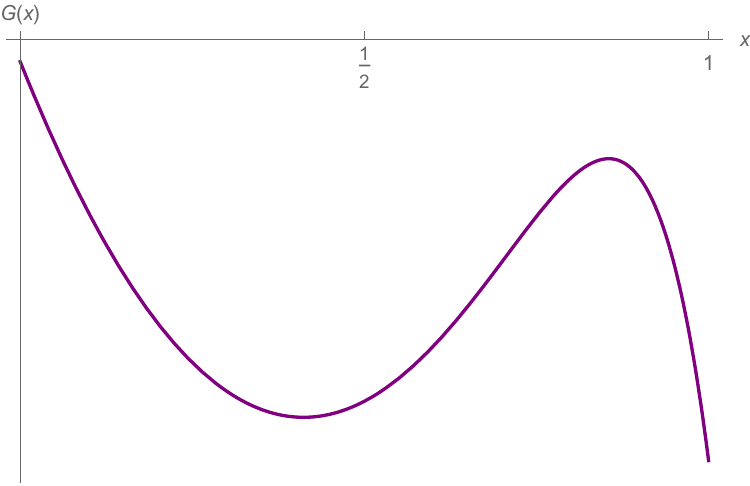}}
    \subfigure[One solution]{\includegraphics[scale=.5]{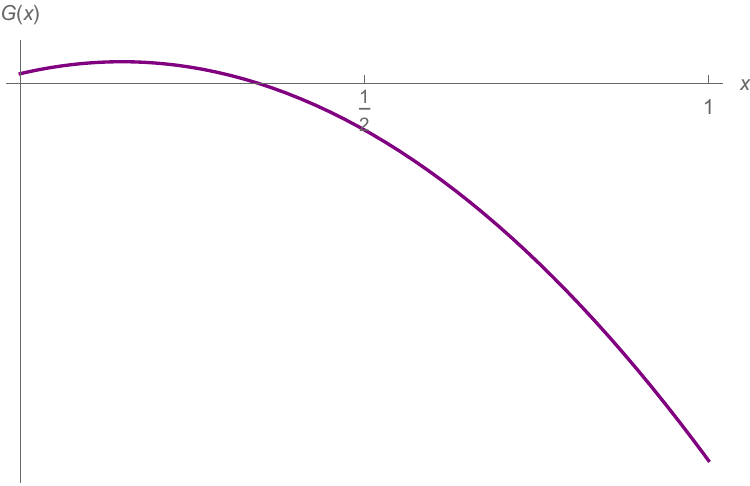}}
      \subfigure[Two solutions]{\includegraphics[scale=.5]{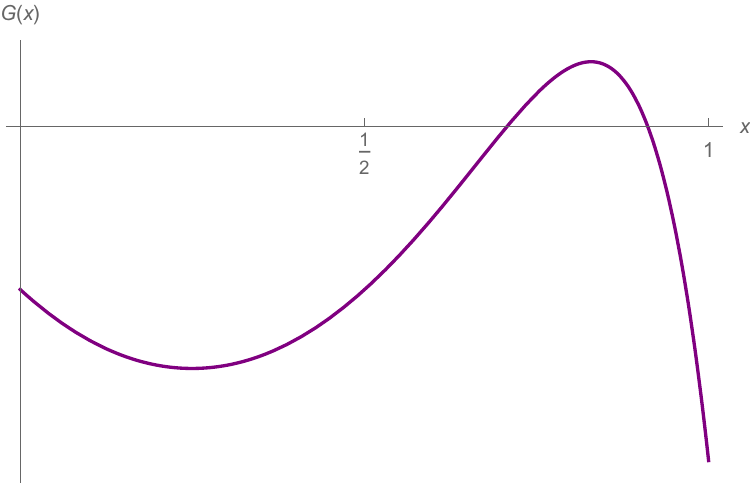}}
        \subfigure[Three solutions]{\includegraphics[scale=.5]{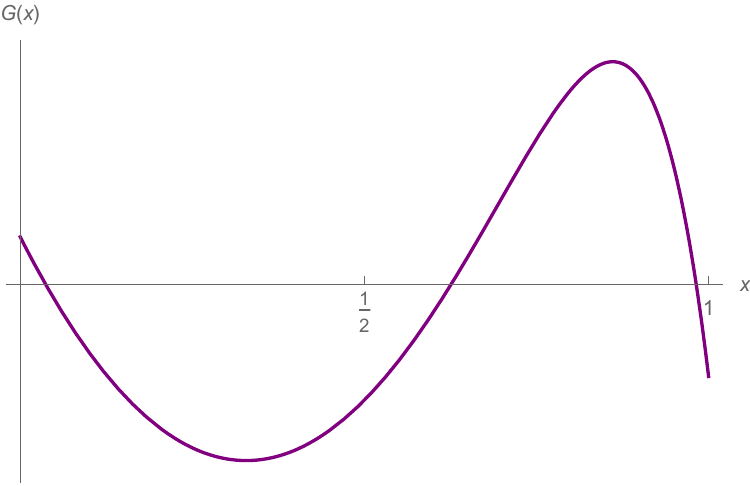}}
    \caption{Possible numbers of solutions when $\omega=0$}
    \label{fig:solutions omega0}
\end{figure}
However, numerics seem to strongly suggest that the equation $F_{\omega_0}=0$ never has four roots. We will proceed to show this analytically.    
    
Assume that the signs of the coefficients are as in (\ref{eq: signs omega0}) and the four roots are $x_1,x_2,x_3$ and $x_4$. We consider two  distinct cases:
    \begin{enumerate}
        \item $d$ is even. We apply Decartes theorem again in order to determine the number of negative roots: after changing $z\mapsto-z$, the coefficients become 
        \begin{equation*}
    \begin{array}{c c c c c}
        b\delta & b\delta q &A &-(B+b\delta)&\mu-b\delta q\\
        +&+&-&-&-
    \end{array}
\end{equation*}
meaning that there must exist one negative root, which we denote by $x_5$. Therefore, the polynomial can be represented as 
\[
F_{\omega_0}(x) = P(x)(x-x_1)(x-x_2)(x-x_3)(x-x_4)(x-x_5),
\]
where $P(x)$ does not have any roots. However, the degree of $P(x)$ is $d-5$, an odd number -- hence, $P(x)$ must have at least one root, which creates a contradiction. 
        \item $d$ is odd. In this case, the distribution of signs for the negative $z$ is 
        \begin{equation*}
         \begin{array}{c c c c c}
        -b\delta & -b\delta q &A &-(B+b\delta)&\mu-b\delta q\\
        -&-&-&-&-
    \end{array}
\end{equation*}
and the polynomial $F_{\omega_0}$ has no negative roots.

Analogously, $F_{\omega_0}(x) = (x-x_1)(x-x_2)(x-x_3)(x-x_4)P(x)$, with $\deg(P(x)) = d-4$, an odd number -- and we arrive at the same contradiction as before, which concludes our proof.
    \end{enumerate}

When $\mu=0$, the  equation turns into 
\[
b \delta  \left(x^d-1\right) (q-x)+\frac{c (2 q-1) r (x-1) x}{d}+c x (q (-2 r x+r+1)+x-1)=0
\]
The left hand side is $-b \delta  q$ when $x=0$ and $-c q (r-1)$ when $x=1$. Since both of these values are less than zero, the equation can only have an even number of roots -- therefore it is zero or 2, as per what we proved above. 
\begin{figure}
    \centering
    \subfigure[Zero Solutions]{\includegraphics[scale=.33]{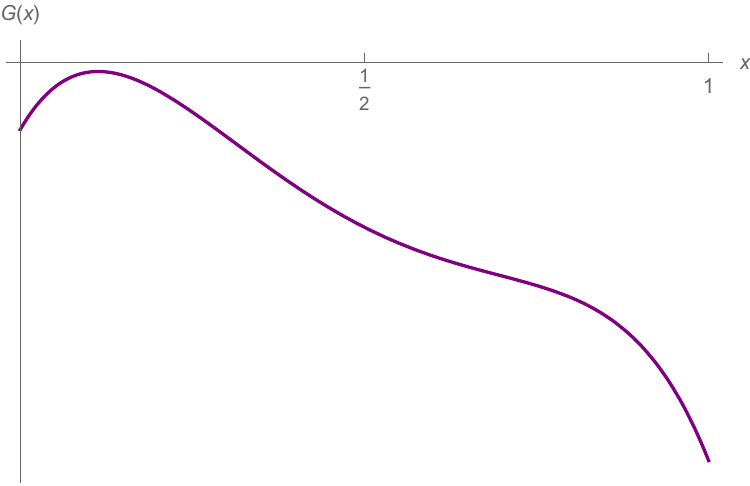}}
    \subfigure[Two Solutions]{\includegraphics[scale=.33]{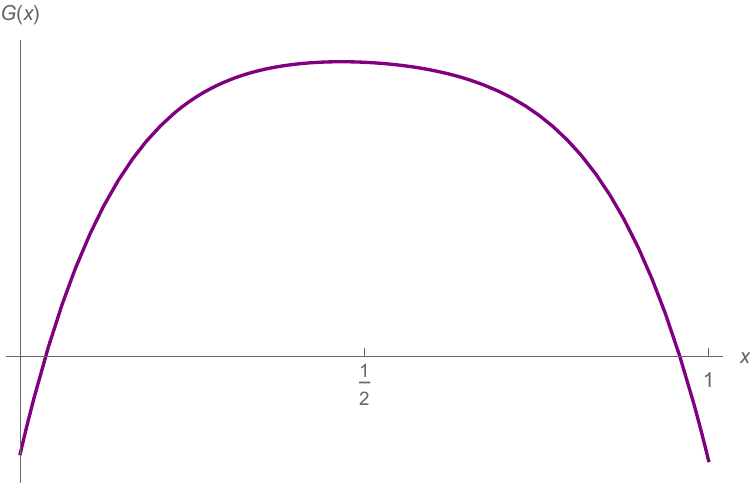}}
    \subfigure[Four Solutions]{\includegraphics[scale=.33]{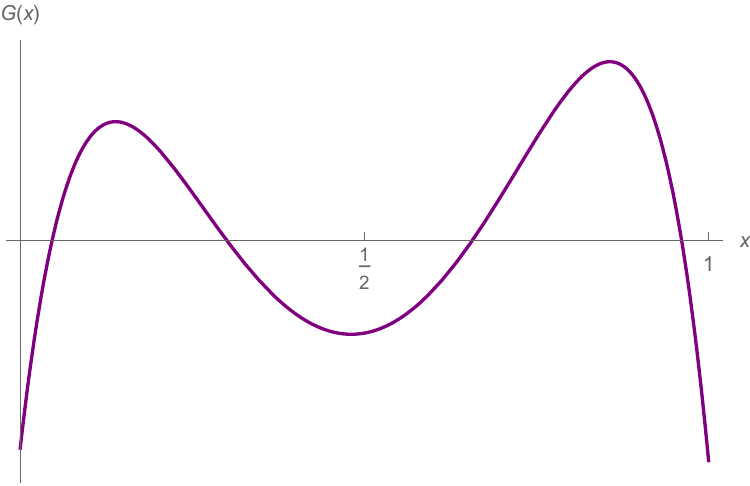}}
    \caption{Varying number of solutions for $\mu=0$}
    \label{fig:mu=0}
\end{figure}
Figure \ref{fig:solutions omega0} illustrates  examples with 0 roots (a), 1 root (b), 2 roots (c) and 3 roots (d).   
 \vskip .2cm
 \noindent$\boldsymbol{\omega = 1
.}$ The variable change $y=1-x$ effectively reduces our case to the previous one; however, some differences exist between the  two.  After substitution, the general form of the equation in $y$ will be 
\begin{small}
    \begin{equation}
        \label{eq: equation in y omega=1}
        a \delta  \left(1-y^d\right) (y-q)+\frac{c (1-y) ((d-1) q r -(1-y) (2 (d-1) q r-d+r)+(1-q) (r-d))}{d}-\mu  (1-2y)=0.
    \end{equation}
\end{small}
For brevity, we denote the coefficients at $y$ and $y^2$  by $B$ and $A$ respectively-- note that they don't have a fixed sign. Analogously to the previous case, we write out the sing changes:
\begin{equation}
\label{eq: change signs y}
         \begin{array}{c c c c c}
        -a\delta & aq\delta  &A &B &q\left(c(1-r)-a\delta\right) -\mu\\
        -&+&?&?&-
    \end{array}
\end{equation}

When $y=0$,  the left-hand side of (\ref{eq: equation in y omega=1}) is $q\left(c(1-r)-a\delta\right) -\mu<0$ and when $y=1$, it is equal to $\mu>0$. Hence, the number of roots is either 1 or 3, and at least one root must exist.  We set out to show that, in actuality, only one solution exists. 

Note that since the sign of the last coefficient is fixed (unlike in the case for $\omega=0$), for three roots to exist we need the coefficients to have exactly the following sings:
\begin{equation}
\label{eq: change signs y2}
         \begin{array}{c c c c c}
        -a\delta & aq\delta  &A &B &q\left(c(1-r)-a\delta\right) -\mu\\
        -&+&-&+&-
    \end{array}
\end{equation}

But we have shown above that the equation of this type can't have four positive roots (even outside of $[0,1]$ interval); therefore, with this arrangement of signs two positive roots exist, with  one of them lying between $0$ and $1$ (note that at least one root must exist). 

Consider all other possible distributions of signs:
\[
 \begin{array}{c c c c c}
        -a\delta & aq\delta  &A &B &q\left(c(1-r)-a\delta\right) -\mu\\
        -&+&+&+&-
    \end{array},
\]
\[
 \begin{array}{c c c c c}
        -a\delta & aq\delta  &A &B &q\left(c(1-r)-a\delta\right) -\mu\\
        -&+&-&-&-
    \end{array},
\]
or 
\[
 \begin{array}{c c c c c}
        -a\delta & aq\delta  &A &B &q\left(c(1-r)-a\delta\right) -\mu\\
        -&+&+&-&-
    \end{array},
\]

all of which have two sign changes, meaning that two or zero positive roots exist. By the same reasoning as above we deduce that these two roots must exist, with one of them belonging to the $[0,1]$-interval. 
\vskip .2cm
\noindent $\boldsymbol{\mu=0.}$
    This assumption turns $G(x)$ into 
    \begin{small}
        \begin{equation}
            \begin{split}
                G(x)_{\mu_0} &= \frac{c}{d}x\left(rq(d-1) + (r-d)(1-q) - x(2rq(d-1) + (r-d))\right)\\&+a\omega\delta(1-q-x)\left(1 - (1-x)^d\right) - b(q-x)\delta(1-\omega)(1-x^d).
            \end{split}
        \end{equation}
    \end{small}
    \begin{figure}
        \centering
      \subfigure[One Solution]{\includegraphics[scale=.5]{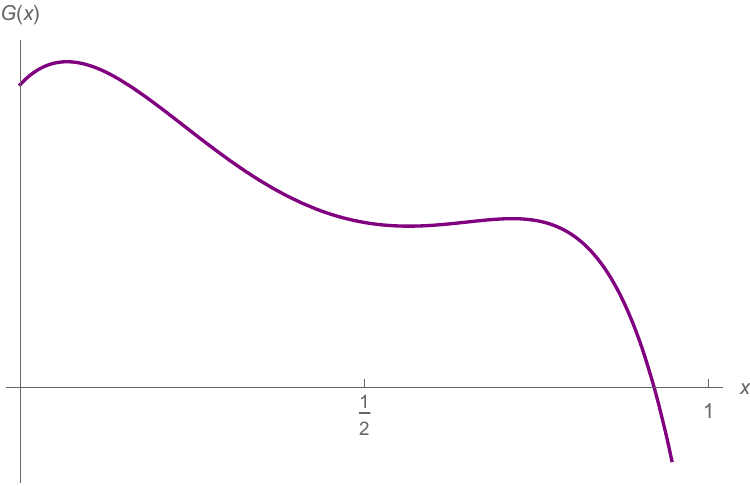}}
      \subfigure[Two Solutions]{\includegraphics[scale=.5]{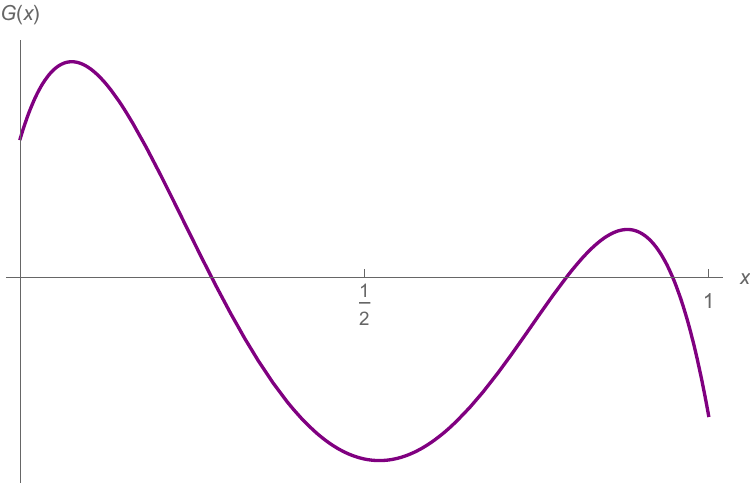}}
        \caption{Different numbers of solutions when $q=0$.}
        \label{fig:q=0}
    \end{figure}
    At the endpoints of the interval 
    \begin{equation*}
        \begin{split}
            G_{\mu_0}(0) &= -b \delta  q (1-\omega )<0,\\
            G_{\mu_0}(1) &= -q (a \delta  \omega +c (r-1))<0;
        \end{split}
    \end{equation*}
      hence, the equation $G_{\mu_0}=0$ has an even number of solutions. Taking Theorem \ref{st: general case} into consideration, this means that $G_{\mu_0}=0$ has $0,2$ or $4$ roots -- see Figure \ref{fig:mu=0} for illustrations.
\noindent

\noindent
$\mathbf{q=0}$.

Substituting $q=0$ turns the function $G(x) $into 
\begin{equation}
    \label{eq:gq0}\
    G_{q_0} := a \delta  (x-1) \omega  \left((1-x)^d-1\right)+b \delta  x (\omega -1) \left(x^d-1\right)+\frac{c (x-1) x (d-r)}{d}+\mu -2 \mu  x.
\end{equation}

Computation shows that $G_{q_0}(0) = \mu$ and $G_{q_0}(1) = -\mu$ -- therefore, the equation 
\[
G_{q_0} = 0
\] 
can only have an odd number of solutions. We know that the number does not exceed 4 in the general case, therefore, we are left with 1 or 3 solutions. 

Finding the appropriate values of the coefficients, we may observe that both cases take place -- see Figure \ref{fig:q=0}
\vskip .2cm
\noindent\textbf{Intersections} All the rest of the cases in the table can be checked by substituting and configurating parameters
\end{proof}

\subsection{Proof of Lemma \ref{st: how g2 behaves with omega and q}}
\begin{proof}[Proof of Lemma \ref{st: how g2 behaves with omega and q}]
    \begin{enumerate}
        \item We are going to demonstrate this explicitly by computing $\frac{\partial G''_1(x)}{\partial\omega}$. 
        \[
        \frac{\partial G''_1(x)}{\partial\omega} = x^{d-2} \left(x - q\frac{d-1}{d+1}\right)-(1-x)^{d-2} \left(1-x - q\frac{d-1}{d+1}\right)
        \]
        One can repeat the proof of Theorem \ref{st: general case} verbatim to ascertain that the equation $\frac{\partial G''_2(x)}{\partial\omega} = 0$ has one root; computation shows that it is located at $x = 1/2$. Since $frac{\partial G''_2}{\partial\omega}(0) = -(d (1-q)+q+1)<0$,  $\frac{\partial G''_2(x)}{\partial\omega}<0$ when $x<1/2$. 
        \item We proceed analogously and compute
        \[
     \frac{\partial G''_1(x)}{\partial q}  =   (d-1) \left((1-\omega ) x^{d-2}+\omega  (1-x)^{d-2}\right)>0 \ \text{for all} \ x\in[0,1].
        \]
    \end{enumerate}
   
\end{proof}
\subsection{Proof of Lemma \ref{lem: q2}}
\begin{proof}[Proof of Lemma \ref{lem: q2}]
    Consider the derivative $q_2'(x)$:
    \begin{small}
        \begin{equation*}
            q_2'(x) =\frac{(d+1) (2 x-1) \left(d^2-4d+2 (x-1) x+4 -\sqrt{d(1-2 x)^2(d-4)+4 (x^2-x)  (x^2- x+4)+4}\right)}{(d-2) (d-1) \sqrt{d(1-2 x)^2(d-4)+4 (x^2-x)  (x^2- x+4)+4}}
        \end{equation*}
    \end{small}
    Clearly, $q_2'(1/2)=0$; we claim that the derivative has no other zeros. To show it, we consider the expression in the parentheses in the numerator:
    \begin{equation*}
    \begin{split}
    d^2-4d+2 (x-1) x+4 -\sqrt{d(1-2 x)^2(d-4)+4 (x^2-x)  (x^2- x+4)+4} &= 0\iff\\
    \left(d^2-4d+2 (x-1) x+4\right)^2- \left(d(1-2 x)^2(d-4)+4 (x^2-x)  (x^2- x+4)+4\right) &=0 \iff \\
    (d-3) (d-2)^2 (d-1) &=0
    \end{split}
    \end{equation*}

The last line is the simplification of the penultimate line. The last statement is clearly false, so $x = 1/2$ remains the only zero. 

It is a minimum of the function, since $q''_1(1/2) = 4 \left(d-\frac{3}{d-2}\right)>0$.

To complete our proof, we compute $q_2(1/2) =\frac{d+1}{2 (d-1)}>\frac12 $.
\end{proof}
\subsection{Proof of Lemma \ref{lem: geo. property}}
\begin{proof}[Proof of Lemma \ref{lem: geo. property}]
The first statement is an immediate corollary of Theorem \ref{st: general case}: for all values of $\mu$ the equation $G(x)=0$ has no more than 4 solutions. The second one follows from Theorem \ref{st: particular cases} -- the equation $G_{\mu_0}=0$ has 0,1,2 or 3 roots.
\end{proof}
\subsection{Proof of Theorem \ref{thm: bifurcation}}

We prove the theorem in a number of steps. 
\begin{lemma}
    The left branch of the graph of $G_{\mu}$ starts above 0, and the right branch "ends" below 0. 
\end{lemma}
\begin{proof}
   It is easily checked that 
    \[
    G_{\mu}(0) = b \delta  q (1-\omega)>0, \ G_{\mu}(1) = -q \left(a \delta  \omega +c (r-1)\right)<0
    \]
\end{proof}
\begin{lemma}
    The function $G_{\mu}(x)$ has at most four maxima and minima (not counting the singularity at $x =1/2$). 
\end{lemma}
\begin{proof}
    The proof is purely computational. 
Calculating the derivative gives 
\begin{small}
    \begin{equation}
    \begin{split}
    \label{eq: first derivative}
        G'_{\mu} &= \frac{d \delta  \left(a \omega  \left((1-x)^{d-1} (-(d (2 x-1) (q+x-1)-2 q x+2 q+x-1))+2 q-1\right)\right)}{ (1-2 x)^2}+\\ & \frac{d\delta\left(b (\omega -1) \left(x^{d-1} \left(-x (2 (d-1) q+d+1)+d q+2 d x^2\right)-2 q+1\right)\right)}{d (1-2 x)^2}+\\  &\frac{c \left(-d q \left(r (1-2 x)^2+1\right)+2 d (x-1) x+d+(2 q-1) r (2 (x-1) x+1)\right)}{d (1-2 x)^2}=:\frac{F(x)}{d(1-2x)^2}
        \end{split}
    \end{equation}
\end{small}
For simplicity, we combine below all the derivatives of the function $F(x)$:
\begin{small}
    \begin{equation}
        \begin{split}
            F'(x) =& \ (2 x-1) \Bigl[\delta  d^2 \bigl(a \omega  (1-x)^{d-2} (d (q+x-1)-q+x-1)+\\ & b (\omega -1) x^{d-2} (d x-d q+q+x)\bigr)+2 c (d(1-2  q r)+(2 q-1) r)\Bigr]\\
            F''(x) =& \  \delta  d^2 \Bigl[(b (\omega -1) x^{d-3} ((d+1) x (2 d x-d+1)-(d-1) q (2 (d-1) x-d+2))-\\& a \omega  (1-x)^{d-3} \left(d^2 (2 x-1) (q+x-1)+d \left(q-4 (q+1) x+2 x^2+2\right)+2 q x-x+1\right)\Bigr]+\\&4 c (-2 d q r+d+(2 q-1) r)\\
            F'''(x) =&\ (d-1) d^2 \delta  \Bigl[a \omega  (1-x)^{d-4} \bigl(d^2 (2 x-1) (q+x-1)+\\&d \left(q(1-6x)+2 x^2-5 x+3\right)+2 (2 q x+q-x+1)\bigr)+\\& b (\omega -1) x^{d-4} ((d+1) x (2 d x-d+2)-(d-2) q (2 (d-1) x-d+3))\Bigr]
        \end{split}
    \end{equation}
\end{small}

Consider the second multiplier in  $F'(x)$, namely,
\begin{small}
   \begin{equation}
       \begin{split}
           f(x):=& \delta  d^2 \left(a \omega  (1-x)^{d-2} (d (q+x-1)-q+x-1)+b (\omega -1) x^{d-2} (d(x-q)+q+x)\right)+\\ &2 c (-2 d q r+d+(2 q-1) r)
       \end{split}
   \end{equation}
\end{small}
Then, 
\begin{small}
    \begin{equation}
        \begin{split}
            f'(x) = (d+1)(d-1) d^2 \delta  \Bigl[-b (1-\omega) x^{d-3} \left(x - q\frac{d-2}{d+1}\right)+a \omega  (1-x)^{d-3} \left(1-x - q\frac{d-2}{d+1}\right)\Bigr],
        \end{split}
    \end{equation}
\end{small}
which coincides with (\ref{eq: G'''(x)}). Recall that the equation $G'''(x)=0$ has a unique solution; hence, $f(x)$ has at most two zeros, $F'(x)$ has three and $F(x)$ at most 4.   
\end{proof}
\begin{lemma}
    The function $G(x)$ can't have all of its four zeros to any one side of $x=1/2$. 
\end{lemma}
\begin{proof}
    First, we note that the exchange $x\mapsto1-x$, $a\mapsto b$, $b\mapsto a$, $\omega\mapsto1-\omega$ preserves the function $G(x)$. Therefore, we need to prove the above statement for $x<1/2$ only. 

To do so, note that since $G'(1/2)=0$, at least one root of $G(x)$ is greater than $1/2$-- this excludes the possibility of all four roots lying to one side.
    
\end{proof}
With these considerations in mind, it is very easy to draw all the possible schematic forms of the graph of $G_{\mu}$, which can be  in Figures \ref{fig: bifurcations 1}, \ref{fig: bifurcations 2} and \ref{fig: bifurcations 3}. Bearing in mind generalisations from the  from the statement of the theorem, all possible shapes can be classified by  the number of critical points on either side of $x=1/2$, asymptotic behaviour and the number of intersections with the $x-$axis. 
\subsection{Conditions for the replicator-mutator dynamics have one or two equilibria}
\label{sec: conditions two solutions}
In this section, we seek conditions for the parameters to satisfy Proposition \ref{st: when one or two solution}. We begin with $G_2(x)$.

The graph of this quadratic function is an upturned parabola if and only if the coefficient at $x^2$ is greater than 0, i.e. if 
\begin{equation}
    \label{eq: coefficient at g2}
    \frac{c(2 (d-1) q r-d+r)}{d}
    >0
\end{equation}
For fixed $d$ and $q$, the area in $c$ and $r$ for which (\ref{eq: coefficient at g2}) holds, is in Figure \ref{fig:conditions in c and r}. 

%The reasoning for the function $G_1(x)$ is slightly more complex. We can simplify our job by demanding that 
 Next we seek the conditions such that the function $G_1$ is concave, that is the second derivative  $G_1''(x)$ thus has to be strictly negative in the $[0,1]$ interval.

Explicitly, we can write
\begin{equation*}
  G_1''(x) = -d \delta  \Bigl((1-\omega ) x^{d-2} ((d+1)x-(d-1)q)+\omega  (1-x)^{d-2} ((d+1)(1-x) - (d-1)q)\Bigr). 
\end{equation*}

Since $G''_1(0) = -a d \delta  \omega  (d (1-q)+q+1)<0$  and $G''_1(1) =- b d \delta  (1-\omega ) (d (1-q)+q+1)<0$, the condition for it assuming only negative values is  the equation 
\begin{equation}
    \label{eq: g''2 has no solutions}
    -\frac{(1-\omega ) x^{d-2} ((d+1)x - (d-1)q)}{\omega }=(1-x)^{d-2} ((d+1) (1-x)-(d-1) q)
\end{equation}
having no solutions. In the light of $G_1''(0)<0$ this will be sufficient for our claim. 

\begin{remark}
    Without loss of generality, we may assume that $x<1/2$. If not, we make an exchange $x\mapsto1-x, \omega\mapsto1-\omega$ and reduce the problem to the previous one. 
\end{remark}
\begin{proposition}
\label{st: how g2 behaves with omega and q}
    When $x<1/2$, the following hold:
    \begin{enumerate}
        \item $G''_1(x)$ is a decreasing function of $\omega$;
        \item  $G''_1(x)$ is an increasing function of $q$.
    \end{enumerate}
\end{proposition}
Note that Equation (\ref{eq: g''2 has no solutions})  can be rewritten as 
\[
KF(x) = F(1-x) \ \text{for some } K<0,
\]
and the graphs of the left hand side and the right hand side need not intersect (see Figure \ref{fig:x and 1-x zoomed in} (a)).

Obviously, they are structured identically (barring the vertical reflection): monotonously increasing from $x=0$ to the maximum at $x = \frac{d-2}{d+1}$, then monotonously decreasing. When the value at the maximum is  made large enough by $K$, the equation has two solutions (see Figure \ref{fig:x and 1-x zoomed in}(b)).

We are concerned with how the $K$-driven bifurcations happen for (\ref{eq: g''2 has no solutions}). The conditions for the two graphs to be  tangent are (at a given point, the functions, as well as their derivatives, must be equal):
\begin{equation}
    \label{eq: when two grahs are tangent}
    \begin{cases}
        (1-\omega ) x^{d-2} (-d q+d x+q+x)=\omega  (1-x)^{d-2} (d (q+x-1)-q+x-1);\\
         (\omega -1) x^{d-3} (-d q+d x+2 q+x)= \omega  (1-x)^{d-3} (d (q+x-1)-2 q+x-1),
    \end{cases}
\end{equation}
The natural next step is to divide the first equation by the second, to obtain
\[
\frac{x (-d q+d x+q+x)}{(d-1) (-d q+d x+2 q+x)}=\frac{(x-1) (d (q+x-1)-q+x-1)}{(d-1) (d (q+x-1)-2 q+x-1)},
\]
an equation that can be solved for $q$, yielding 
\begin{small}
\begin{equation}
    \begin{split}
    \label{eq: conditions on q}
       q_1&= \frac{(d+1) \left(d-2-2x (x-1) +\sqrt{d(1-2 x)^2(d-4)+4 (x^2-x)  (x^2- x+4)+4}\right)}{2 (d-2) (d-1)},\\
       q_2&= \frac{(d+1) \left(d-2-2x (x-1) -\sqrt{d(1-2 x)^2(d-4)+4 (x^2-x)  (x^2- x+4)+4}\right)}{2 (d-2) (d-1)}.
    \end{split}
\end{equation}
\end{small}

\begin{lemma}
\label{lem: q2}
    The solution $q_2$ does not belong to $\left[0,\frac12\right]$-interval. 
\end{lemma}
Only the solution $q_1$ of (\ref{eq: conditions on q}) is admissible. Substituting it into the first equation of (\ref{eq: when two grahs are tangent}) yields 
\begin{small}
    \begin{equation*}
        \begin{split}
            &\frac{(\omega -1) x^{d-2} \left(\sqrt{d^2 (1-2 x)^2-4 d (1-2 x)^2+4 (x-1) x ((x-1) x+4)+4}-2 d x+d-2 (x-3) x-2\right)}{2 (d-2)}\\&-\frac{\omega  (1-x)^{d-2} \left(\sqrt{d^2 (1-2 x)^2-4 d (1-2 x)^2+4 (x-1) x ((x-1) x+4)+4}+d (2 x-1)-2 x (x+1)+2\right)}{2 (d-2)}\\&=0. 
        \end{split}
    \end{equation*}
\end{small}
This equation  can be solved for $\omega$ to give
\begin{small}
\begin{equation}
    \label{eq: conditions for omea}
    \omega = \frac{2 (x-1)^3 x^d}{x(1-x)^d\left( \sqrt{d(d-4)(1-2 x)^2+4 (x^2-x)  (x^2-x+4)+4}+(d-2)  (2 x-1)\right) +2 (x-1)^3 x^d}.
\end{equation}
\end{small}
We have two functions, (\ref{eq: conditions on q}) and (\ref{eq: conditions for omea}), which tell us the values that $\omega$ and $q$ must assume in order for the two graphs from (\ref{eq: g''2 has no solutions}) to be tangent.
\subsection{Proof of Lemma \ref{lem: approximation}}
\begin{proof}[Proof of Lemma \ref{lem: approximation}]
When referring to approximation below, we mean the following: a continuously differentiable function $f_1(x)$ approximates another continuously differentiable function $f_2(x)$ on some interval $[a,b]$ if $|f_1(x)-f_2(x)|<\epsilon$ for some fixed small $\epsilon$. 

To prove Lemma \ref{lem: approximation}, we require a few observations.
\begin{enumerate}

    \item If an intersection of a graph of the function $f(x)$ with the $x-$axis is transversal, the derivative $f'(x)$ at the point of intersection is separated from 0 by some margin $\epsilon$.
    \item  Suppose that we have a function $F(x)$ with a simple zero at the point $x^{\ast}$ and another function, $f(x)$, that $F(x)$ approximates in some neighbourhood of the point $x^{\ast}$ with enough precision. This entails that $f(x)$ has a zero near $x^{\ast}$. 
    
    In order for the function $f(x)$ to have additional zeros near $x^{\ast}$, its derivative needs to turn 0 in some neighbourhood of $x^{\ast}$. 
\end{enumerate}
Now we are ready to prove Lemma \ref{lem: approximation}. Denote the function 
\begin{equation}
\begin{split}
    \label{eq: approximating function}
 g_1(x)-g_2(x) &= x\delta\left(b (1-\omega)-a \omega \right)+\delta(a (1-q) \omega -b q (1-\omega )) \\&- \left( x^2 c(2  q r-1)+x (c (1-q -qr)+2 \mu )-\mu=\right)\\
   &=:g(x).
   \end{split}
\end{equation}
Then, since $(1-x)^d>x^d$ when $x<\frac12$ and $(1-x)^d<x^d $ when $x>\frac12$, 
\[
G(x) - g(x)<\begin{cases}
-(1-x)^{d-1} (a \omega  (d (q+x-1)+x-1)+b (\omega -1) (d q-(d+1) x))\\ +\frac{c (2 q-1) r (x-1) x}{d}\ \text{when}  \ x\le1/2\\
\\

x^{d-1} (-a \omega  (d (q+x-1)+x-1)-b (\omega -1) (d q-(d+1) x))\\
+\frac{c (2 q-1) r (x-1) x}{d}\ \text{when}  \ x>1/2\\

    \end{cases}
\]
It can easily be seen that the expression above goes to $0$ when $d\to\infty$ and $x$ sufficiently removed from 0 and 1. Note that the derivative  $G'(x)-g'(x)$ behaves similarly.

Therefore, in order to show that the number of zeros of $G(x)$ coincides with that of $g(x)$, it is sufficient to show that with the appropriate choice of other parameters (1) at the zeros of $g(x)$ its derivative, $g'(x)$ is not close to 0, (2) at the single zero of the linear function $g'(x)$ the function $g(x)$ is not close to 0 and (3) the solutions of $g(x)=0$ are not infinitesimally close to 0 or 1.
\begin{enumerate}
    \item The zeros of $g(x)$ have the form 
    \begin{small}
    \begin{equation}
    \begin{split}
    x_{1,2} &= \frac{\pm\sqrt{(-\omega  (a+b)+b+c q r+c q-c+2 \mu )^2-4 (c-2 c q r) (-a (q-1) \omega +b q (\omega -1)-\mu )}}{2 c (2 q r-1)}\\ &+\frac{b (1-\omega )-a \omega+c( q (r+ 1)-1)+2 \mu }{2 c (2 q r-1)},
    \end{split}
    \end{equation}
\end{small}
and the values of the derivative at these points are
\begin{small}
\begin{equation}
    \begin{split}
       g'(x_{1,2})= \pm\sqrt{(b(1-\omega)-\omega a+c( q r+ q-1)+2 \mu )^2-4 c(1-2  q r) (b q (\omega -1) -a (q-1) \omega-\mu )}.
   \end{split}
\end{equation}
\end{small}
The expression under the square root is a quadratic function in $\mu$ with the top coefficient equal to 4; therefore, the square root of it can be made arbitrarily big by the choice of the parameters. 
\item Computations show that 
\[
g'\left(-\frac{-\omega  (a+b)+b+c q r+c q-c+2 \mu }{c (2-4 q r)}\right)=0.
\]
At this point, the value of $g(x)$ is equal to 
\begin{small}
    \begin{equation}
        \begin{split}
    & \frac{2 b (\omega -1) (a \omega +c q (4 q-1) r-3 c q+c-2 \mu )+2 a c \omega  ((3-4 q) q r+q-1)}{4 c (2 q r-1)} \\
     &+\frac{(a \omega -2 \mu )^2+b^2 (\omega -1)^2+c^2 (q r+q-1)^2-4 c \mu  q (r-1)}{4 c (2 q r-1)}
        \end{split}
    \end{equation}
\end{small}
Again a quadratic function in $\mu$, this expression can be made arbitrarily big by  demanding $4 c (2 q r-1)>0$ and increasing the value of $\mu$.
\item This was actually covered in the first entry. 
\end{enumerate}    
\end{proof}
\section*{Acknowledgement}
M.H.D and N.B. are supported by EPSRC (grant EP/Y008561/1). T.A.H. is supported by EPSRC (grant
EP/Y00857X/1). 
\bibliographystyle{alpha}
%\bibliography{ref}
\newcommand{\etalchar}[1]{$^{#1}$}

\end{document}